\documentclass[11pt,reqno]{amsart}

\scrollmode
\usepackage{tikz}
\usepackage{tikz-cd}

\usepackage{times}
\usepackage{amsmath,graphicx}
\usepackage{latexsym}
\usepackage{amscd,amsthm,amssymb}
\usepackage[all]{xy}

\newtheorem{thm}{Theorem}[section]
\newtheorem{prop}[thm]{Proposition}
\newtheorem{lem}[thm]{Lemma}
\newtheorem{cor}[thm]{Corollary}

\theoremstyle{definition}

\newtheorem{defn}[thm]{Definition}
\newtheorem*{conv}{Convention}

\theoremstyle{remark}
\newtheorem{rem}[thm]{Remark}

\numberwithin{equation}{section}

\title{A lift of the Seiberg--Witten equations to Kaluza--Klein $5$-manifolds}
\author{M.~J.~D.~Hamilton}
\address{      Fachbereich Mathematik\\
               Universit\"at Stuttgart\\
               Pfaffenwaldring 57\\
               70569 Stuttgart\\
               Germany}
\email{mark.hamilton@math.lmu.de}
\date{\today}

\DeclareRobustCommand\longtwoheadrightarrow
     {\relbar\joinrel\twoheadrightarrow}

\begin{document}

\date{\today}

\begin{abstract}
We consider Riemannian $4$-manifolds $(X,g_X)$ with a $\mathrm{Spin}^c$-structure and a suitable circle bundle $Y$ over $X$ such that the $\mathrm{Spin}^c$-structure on $X$ lifts to a spin structure on $Y$. With respect to these structures a spinor $\phi$ on $X$ lifts to an untwisted spinor $\psi$ on $Y$ and a $\mathrm{U}(1)$-gauge field $A$ for the $\mathrm{Spin}^c$-structure can be absorbed into a Kaluza--Klein metric $g_Y^A$ on $Y$. We show that irreducible solutions $(A,\phi)$ to the Seiberg--Witten equations on $(X,g_X)$ for the given $\mathrm{Spin}^c$-structure are equivalent to irreducible solutions $\psi$ of a Dirac equation with cubic non-linearity on the Kaluza--Klein circle bundle $(Y,g_Y^A)$. As an application we consider solutions to the equations in the case of Sasaki $5$-manifolds which are circle bundles over K\"ahler--Einstein surfaces.
\end{abstract}

\maketitle

\section{Introduction}\label{sect:intro}
Suppose that $M$ is a $4$-manifold with a Lorentz metric $g_M$ and an electromagnetic gauge field, i.e.~a $1$-form $A\in\Omega^1(M)$. The original Kaluza--Klein ansatz \cite{Kal}, \cite{Kle} is to consider the $5$-manifold $Y=M\times S^1$ and combine $g_M$ and $A$ to a Lorentz metric $g_Y$ on $Y$ which is invariant under the circle action. The $5$-dimensional vacuum Einstein field equations for $g_Y$ (i.e.~vanishing of the Ricci tensor) then imply the $4$-dimensional Einstein field equations for $g_M$ with electromagnetic source and the Maxwell equation for $A$.

Let $(X,g_X)$ be a smooth, closed, oriented, Riemannian $4$-manifold. We choose a $\mathrm{Spin}^c$-structure $\mathfrak{s}_X^c$ on $X$ with associated spinor bundle $S_X^c$ and characteristic line bundle $L$. Locally the spinor bundle $S_X^c$ is a tensor product of a standard spinor bundle and a square root $L^{\scriptscriptstyle\frac{1}{2}}$, i.e.~a twisted spinor bundle. Since $X$ does not necessarily admit a spin structure, the standard spinor bundle and the square root may not exist globally on $X$. However, $\mathrm{Spin}^c$-structures always exist on closed, oriented $4$-manifolds. 

The Seiberg--Witten equations \cite{SW1}, \cite {SW2}, \cite{W} are partial differential equations for a pair $(A,\phi)$, consisting of a Hermitian connection $A$ on $L$ and a positive Weyl spinor $\phi\in\Gamma(S_X^{c+})$. These equations can be used to define the Seiberg--Witten invariants of $X$ which have numerous applications to the differential geometry, symplectic geometry and topology of $4$-manifolds.

Let $\pi\colon Y\rightarrow X$ be the principal circle bundle with Euler class $e(Y)=c_1(L)$. For this choice of Euler class the $\mathrm{Spin}^c$-structure $\mathfrak{s}_X^c$ on $X$ lifts to a {\em spin} structure $\mathfrak{s}_Y$ on $Y$. For the associated spinor bundles, every spinor $\phi\in\Gamma(S_X^c)$ has a canonical lift to an untwisted spinor $\psi=Q(\phi)\in\Gamma(S_Y)$.

We can think of the connection $A$ as a $\mathrm{U}(1)$-connection $A\in\Omega^1(Y,i\mathbb{R})$ on the principal bundle $Y$ and define the Kaluza--Klein metric
\begin{equation*}
g_Y=g_Y^A=\pi^*g_X-A\otimes A
\end{equation*}
on $Y$, which is a Riemannian metric, because $A$ is assumed to be imaginary-valued.
\begin{thm}\label{thm:main thm n=1}
We consider the Seiberg--Witten equations on $X$:
\begin{equation}\label{eqn:SW eqns intro}
\begin{split}
D_A^X\phi&=0\\
F_A^+&=\sigma(\phi,\phi).
\end{split} 
\end{equation}
Let $\phi\in \Gamma(S_X^{c+})$ be a positive Weyl spinor on $X$ and $\psi =Q(\phi)\in \Gamma(S_Y)$ the lift to $Y$. If $(A,\phi)$ is a solution to the Seiberg--Witten equations \eqref{eqn:SW eqns intro}, then $\psi$ is a solution to the equation
\begin{equation}\label{eqn:DY eqn n=1}
D^Y\psi=\frac{1}{8}|\psi|^2\psi-\frac{1}{2}\psi.
\end{equation}
Here $D^Y$ is the Dirac operator on $S_Y$ for the Kaluza--Klein metric $g_Y^A$. If $\phi$ does not vanish identically on $X$, then the converse holds as well, i.e.~if $\psi=Q(\phi)$ is a solution to equation \eqref{eqn:DY eqn n=1}, then $(A,\phi)$ is a solution to the Seiberg--Witten equations \eqref{eqn:SW eqns intro}.
\end{thm}
Theorem \ref{thm:main thm n=1} is a special case of the more general Theorem \ref{thm:main thm all n}. In particular, the spin structure on $Y$ in Theorem \ref{thm:main thm n=1} is of odd type, cf.~Section \ref{sect:spin spinc kk bundles}. If $X$ is a spin manifold (i.e.~admits a spin structure), we can choose the spin structure on $Y$ to be even. We also consider the generalization to Kaluza--Klein circle bundles with fibres of length $2\pi r$, where $r$ can vary over the base manifold $X$, see Section \ref{sect:KK bundles fibres length 2pir}. The overall sign on the right hand side of equation \eqref{eqn:DY eqn n=1} depends on the relation between Clifford multiplication in dimensions $4$ and $5$, cf.~Remark \ref{rem:depend sign ident Cliff 4 5 dim}.
\begin{rem}
Pairs $(A,\phi)$ and spinors $\psi$ are called irreducible if $\phi$ and $\psi$ are not identically zero on $X$ and $Y$, respectively. The theorem implies that $(A,\phi)$ is an irreducible solution to the Seiberg--Witten equations \eqref{eqn:SW eqns intro} on $(X,g_X)$ if and only if $\psi=Q(\phi)$ is an irreducible solution to equation \eqref{eqn:DY eqn n=1} on the Kaluza--Klein circle bundle $(Y,g_Y^A)$. Reducible solutions $(A,0)$ to the Seiberg--Witten equations are given by connections $A$ so that $F_A^+=0$. The trivial spinor $\psi\equiv 0$ is a solution to equation \eqref{eqn:DY eqn n=1} without a condition on $A$.
\end{rem}
The pullback of a $\mathrm{Spin}^c$-structure $\mathfrak{s}_X^c$ to a spin structure $\mathfrak{s}_Y$ on the circle bundle $Y\rightarrow X$ defined by the characteristic line bundle of $\mathfrak{s}_X^c$ has been discussed in \cite{Mor}. The proof of Theorem \ref{thm:main thm n=1} and Theorem \ref{thm:main thm all n} depends on a calculation in \cite{AB} of the Dirac operator $D^Y$ on lifted spinors $\psi=Q(\phi)$, see Proposition \ref{prop:formula Dirac DY spin and non-spin}. The implication in the case $\phi\not\equiv 0$ from equation \eqref{eqn:DY eqn n=1} to the Seiberg--Witten equations \eqref{eqn:SW eqns intro} follows from the unique continuation property of Dirac operators, cf.~Proposition \ref{prop:equiv SW multiplied phi}.

Both Seiberg--Witten equations on $X$ combine to a single equation on $Y$ because the Dirac operator $D^Y$ maps the lift of a positive Weyl spinor on $X$ to the lift of a mixed spinor. Similarly the map (with Clifford multiplication in the second entry)
\begin{align*}
f\colon\Gamma(S_X^{c+})&\longrightarrow \Gamma(S_X^{c-})\oplus \Gamma(S_X^{c+})\\
\phi&\longmapsto \left(D_A^X\phi,(F_A^+-\sigma(\phi,\phi))\cdot\phi\right),
\end{align*}
has mixed image. Equation \eqref{eqn:DY eqn n=1} is a lift of the equation $f(\phi)=0$. Note that the quadratic non-linearity $\sigma(\phi,\phi)$ in the Seiberg--Witten equations becomes the non-linearity $|\psi|^2$.
\begin{rem}
In Seiberg--Witten theory one often considers the space of all solutions $(A,\phi)$ to equations \eqref{eqn:SW eqns intro} where the $\mathrm{Spin}^c$-structure $\mathfrak{s}_X^c$ and Riemannian metric $g_X$ are considered as fixed parameters. For the corresponding space of solutions $\psi$ of equation \eqref{eqn:DY eqn n=1} on $Y$ the Riemannian metric $g_Y^A$ varies, because $g_Y^A$ depends on $A$. For a given metric $g_X$, the Kaluza--Klein construction can be viewed as an embedding of the space of connections on $Y$ into the space of Riemannian metrics on $Y$.
\end{rem}
\begin{rem} The implications of the invariance of the Seiberg--Witten equations under gauge transformations and charge conjugation are discussed in Section \ref{sect:gauge transform charge conj}.
\end{rem}
\begin{rem}
Equation \eqref{eqn:DY eqn n=1} makes sense for spinors $\psi\in\Gamma(S_Y)$ on any oriented Riemannian $5$-manifold $(Y,g_Y)$ (or, more generally, an $(n+1)$-manifold of arbitrary dimension $n+1$) with a spin structure. The solutions of an equation of type
\begin{equation*}
D^Y\psi=-\kappa|\psi|^2\psi+m\psi
\end{equation*}
with constants $\kappa,m\in\mathbb{R}$ on $(Y,g_Y)$ arise as critical points of the functional (see \cite{I})
\begin{equation*}
S[\psi]=\int_Y\mathcal{L}[\psi]\,\mathrm{dvol}_{g_Y}
\end{equation*}
for the quartic Lagrangian
\begin{equation*}
\mathcal{L}[\psi]=\langle \psi,D^Y\psi\rangle -m|\psi|^2+\frac{1}{2}\kappa|\psi|^4.
\end{equation*}
This Lagrangian is related to the four-fermion interaction studied in the $2$-dimensional Gross--Neveu model \cite{GN}.
\end{rem}
As an application we consider in Section \ref{sect:Sasaki} the Boothby-Wang construction of circle bundles over closed K\"ahler--Einstein surfaces $(X,g_X,J,\omega)$ with $\mathrm{Ric}_{g_X}=\lambda g_X$ and Einstein constant $\lambda\neq 0$. There are canonical $\mathrm{Spin}^c$-structures on $X$ with characteristic line bundles $K$ and $K^{-1}$, defined by the complex structure $J$. The (perturbed) Seiberg--Witten equations for both $\mathrm{Spin}^c$-structures have canonical solutions $(A_0,\phi_0)$ with $|\phi_0|\equiv\mathrm{const}$. In this situation the circle bundle $(Y,g_Y)$ is a Sasaki $\eta$-Einstein $5$-manifold (if the $S^1$-fibres have a suitable constant length, depending on $|\lambda|$) and the spinors $\phi_0$ lift to eigenspinors $\psi$ of the Dirac operator $D^Y$. If the Einstein constant $\lambda$ is negative and $g_X$ normalized so that $\lambda=-4$, the spinors $\psi$ are harmonic, i.e.~in the kernel of $D^Y$. If $\lambda>0$ and $g_X$ normalized so that $\lambda=6$, the $5$-manifold $(Y,
g_Y)$ is Sasaki--Einstein and the spinors $\psi$ are Killing spinors.  
\begin{rem}
Other generalizations of the Seiberg--Witten equations to $5$-manifolds can be found in several references, including \cite{DB}, \cite {KLW}, \cite{Pan}. 
\end{rem}
\begin{conv}\label{convention}
In the following all manifolds are smooth, non-empty, connected and oriented if not stated otherwise. 
\end{conv}

\section{Kaluza--Klein circle bundles}\label{sect:kk circle bundles}
Let $(X^4,g_X)$ be a closed, oriented, Riemannian $4$-manifold and 
\begin{equation*}
\pi\colon Y^5\longrightarrow X^4
\end{equation*}
the oriented principal $S^1$-bundle with Euler class $e\in H^2(X;\mathbb{Z})$ and group action $S^1\times Y\rightarrow Y$. We denote by $K$ the vector field on $Y$ along the fibres, given by the infinitesimal action of a fixed element in the Lie algebra $\mathfrak{u}(1)$ of $S^1$, normalized such that the flow of $K$ has period $2\pi$. 

We define a Riemannian metric $g_Y$ on $Y$ with the following Kaluza--Klein ansatz \cite{Bour}: Let $A\in\Omega^1(Y,i\mathbb{R})$ be a $\mathrm{U}(1)$-connection on the principal bundle $Y\rightarrow X$. This means that $A$ is invariant under the circle action,
\begin{equation*}
L_KA=0,
\end{equation*}  
and normalized such that
\begin{equation*}
A(K)\equiv i.
\end{equation*}
The curvature $2$-form $F_A\in \Omega^2(X,i\mathbb{R})$ satisfies
\begin{equation*}
\pi^*F_A=dA.
\end{equation*}
\begin{defn}\label{defn:KK metric}
The Riemannian metric
\begin{equation*}
g_Y=g_Y^A=\pi^*g_X-A\otimes A
\end{equation*}
is called the {\em Kaluza--Klein metric} on $Y$. We call the principal circle bundle $Y\rightarrow X$ over a Riemannian manifold $(X,g_X)$ together with the Kaluza--Klein metric $g_Y$ for a connection $A$ on $Y$ a {\em Kaluza--Klein circle bundle}.
\end{defn}
The metric $g_Y$ has the following properties:
\begin{itemize}
\item The horizontal bundle $H=\mathrm{ker}\,A$ is orthogonal to the circle fibres.
\item $g_Y|_H=\pi^*g_X$.
\item $g_Y(K,K)=1$, i.e.~the circle fibres have length $2\pi$.
\item $K$ is a Killing vector field for $g_Y$, i.e.~$L_Kg_Y=0$.
\end{itemize}
The Kaluza--Klein metric $g_Y$ on $Y$ is completely characterized by these properties.

Let $E\rightarrow X$ be the complex line bundle (unique up to isomorphism) with $c_1(E)=e$. We can define $E$ as the complex line bundle associated to $Y$ via the standard representation of $S^1$ on $\mathbb{C}$. Then $E$ has a Hermitian metric and $Y$ can be thought of as the unit circle bundle in $E$. The connection $A$ on $Y$ induces a connection (covariant derivative) on $E$ which is compatible with the Hermitian metric and vice versa.
\begin{lem}\label{lem:only bundles trivial lift Ln}
The kernel of $\pi^*\colon H^2(X;\mathbb{Z})\longrightarrow H^2(Y;\mathbb{Z})$ is $\mathbb{Z}e$.
\end{lem}
\begin{proof}
Recall from Convention \ref{convention} that $X$ is assumed non-empty and connected, hence $H^0(X;\mathbb{Z})\cong\mathbb{Z}$. The claim is then equivalent to exactness of the Gysin sequence
\begin{equation*}
H^0(X;\mathbb{Z})\stackrel{\cup e}{\longrightarrow}H^2(X;\mathbb{Z})\stackrel{\pi^*}{\longrightarrow} H^2(Y;\mathbb{Z})\longrightarrow\ldots
\end{equation*}
at $H^2(X;\mathbb{Z})$.
\end{proof}
The lemma implies:
\begin{prop}
Let $\pi\colon Y\rightarrow X$ be the principal circle bundle with Euler class $e$ and $L\rightarrow X$ a complex line bundle on $X$. Then $\pi^*L$ is a trivial line bundle if and only if $c_1(L)\in\mathbb{Z}e$.
\end{prop}
Let $\mathfrak{s}_X^c$ be a $\mathrm{Spin}^c$-structure on $X$ with characteristic line bundle $L$. The heuristic idea to define the lift of $\mathfrak{s}_X^c$ to a spin structure on $Y$ is to choose a Kaluza--Klein circle bundle $Y\rightarrow X$ with Euler class $e$, so that $c_1(L)\in\mathbb{Z}e$. We can then lift $\mathfrak{s}_X^c$ to a $\mathrm{Spin}^c$-structure $\mathfrak{s}_Y^c$ on $Y$ with characteristic line bundle $\pi^*L$. Since this bundle is trivial, it corresponds to a spin structure $\mathfrak{s}_Y$ on $Y$. The details will be worked out in Sections \ref{sect:spin and spinc str on mfds}--\ref{sect:spinors KK} below.

\section{Spin structures and $\mathrm{Spin}^c$-structures on manifolds}\label{sect:spin and spinc str on mfds}

We collect some background material on spin and $\mathrm{Spin}^c$-structures (more details can be found, for example, in \cite{BHMMM}, \cite{F}, \cite{LM}, \cite{Morgan}). Let $(M,g_M)$ be an oriented Riemannian manifold of dimension $n$.

\subsection{Existence and classification}
We consider the double covering
\begin{equation*}
\lambda\colon\mathrm{Spin}(n)\longrightarrow\mathrm{SO}(n),
\end{equation*}
which is the universal covering if $n\geq 3$. With respect to the Riemannian metric $g_M$ the set of oriented orthonormal frames in $TM$ forms a principal bundle $P_{\mathrm{SO}}(M)$ over $M$ with action
\begin{equation*}
P_{\mathrm{SO}}(M)\times\mathrm{SO}(n)\longrightarrow P_{\mathrm{SO}}(M).
\end{equation*}
A spin structure on $M$ is a principal bundle
\begin{equation*}
P_{\mathrm{Spin}}(M)\times\mathrm{Spin}(n)\longrightarrow P_{\mathrm{Spin}}(M)
\end{equation*}
over $M$ together with a smooth bundle map $\mathfrak{s}\colon P_{\mathrm{Spin}}(M)\rightarrow P_{\mathrm{SO}}(M)$ which is equivariant with respect to the homomorphism $\lambda$, i.e.~the following diagram commutes:
\begin{equation*}
\begin{tikzcd}
P_{\mathrm{Spin}}(M)&[-3em]\times\ar[dd, "\mathfrak{s}\times \lambda"]&[-3em]\mathrm{Spin}(n)\ar[r] & P_{\mathrm{Spin}}(M)\ar[dd, "\mathfrak{s}"]\ar[rd] &\\
& & &&M \\
P_{\mathrm{SO}}(M)&\times& \mathrm{SO}(n)\ar[r] & P_{\mathrm{SO}}(M)\ar[ru] &
\end{tikzcd}
\end{equation*}
It follows that the spin structure $\mathfrak{s}$ is a double covering which restricts on each fibre of the principal bundles to a double covering equivalent to $\lambda$.

We denote by
\begin{equation*}
\kappa_n\colon\mathrm{Spin}(n)\longrightarrow U(\Delta_n)    
\end{equation*}
the unitary complex (Dirac) spinor representation and by
\begin{equation*}
S_M=P_{\mathrm{Spin}}(M)\times_{\kappa_n}\Delta_n    
\end{equation*}
the associated Hermitian (Dirac) spinor bundle of $\mathfrak{s}$.

The spinor representation is the restriction of an irreducible representation of the complex Clifford algebra $\mathbb{C}\mathrm{l}(n)$ on $\Delta_n$. The embedding $\mathbb{R}^n\subset \mathbb{C}\mathrm{l}(n)$ and the vector space isomorphism $\Lambda^*\mathbb{R}^n\otimes\mathbb{C}\cong\mathbb{C}\mathrm{l}(n)$ yield bilinear, fibrewise Clifford multiplications
\begin{align*}
TM\times S_M&\longrightarrow S_M,\quad (V,\phi)\longmapsto \gamma(V)\phi=V\cdot\phi\\
\Lambda^*TM\times S_M&\longrightarrow S_M,\quad (\tau,\phi)\longmapsto \gamma(\tau)\phi=\tau\cdot\phi
\end{align*}
which are compatible with the Hermitian bundle metric on $S_M$ and extend via $g_M$ to the dual bundles $T^*M$ and $\Lambda^*T^*M$ and to the complexifications of these bundles. Clifford multiplication with tangent vectors $V,W\in TM$ satisfies
\begin{equation*}
V\cdot W\cdot\phi+W\cdot V\cdot\phi=-2g_M(V,W)\phi\quad\forall \phi\in S_M.
\end{equation*}
The Levi--Civita connection of the Riemannian metric $g_M$ defines a Clifford connection $\nabla^M$ on $S_M$. Together with Clifford multiplication by tangent vectors we get the Dirac operator 
\begin{equation*}
D^M\colon \Gamma(S_M)\longrightarrow \Gamma(S_M),
\end{equation*}
a first order linear differential operator.

If $n=\dim M$ is even, there is a unique irreducible representation of $\mathbb{C}\mathrm{l}(n)$ on $\Delta_n$ and the spinor representation of $\mathrm{Spin}(n)$ decomposes into the direct sum of two irreducible Weyl representations on subspaces $\Delta_n^+$, $\Delta_n^-$ of half dimension. There is a corresponding decomposition of the spinor bundle
\begin{equation}\label{eqn:Weyl spinor split}
S_M=S_M^+\oplus S_M^-    
\end{equation}
into Weyl spinor bundles. Clifford multiplication with a tangent vector and the Dirac operator map 
\begin{equation}\label{eqn:Clifford Dirac Weyl split}
\begin{split}
TM\times S_M^{\pm}&\longrightarrow S_M^{\mp}\\
D^M\colon \Gamma(S_M^{\pm})&\longrightarrow \Gamma(S_M^{\mp}).
\end{split}
\end{equation}
If $n=\dim M$ is odd, there are two inequivalent irreducible representation of $\mathbb{C}\mathrm{l}(n)$ on $\Delta_n$ and the restrictions to $\mathrm{Spin}(n)$ are equivalent and irreducible.

In dimension $4$ and $5$, the cases we are most interested in, both $\Delta_4$ and $\Delta_5$ can be identified as a complex vector space with $\Delta=\mathbb{C}^4$. The Weyl spinor spaces $\Delta_4^\pm$ are given as the $\pm 1$-eigenspaces of 
\begin{equation}\label{eqn:omegaC dim 4}
\omega_{\mathbb{C}}=-e_1e_2e_3e_4\in\mathbb{C}\mathrm{l}(4)
\end{equation}
for an oriented orthonormal basis $e_1,e_2,e_3,e_4$ of $\mathbb{R}^4$. Clifford multiplications $\cdot_4$ and $\cdot_5$ on $\Delta$ are related as follows: Identify $\mathbb{R}^4=\{0\}\times\mathbb{R}^4\subset\mathbb{R}^5$ and let $e_0,e_1,e_2,e_3,e_4$ be an orthonormal basis of $\mathbb{R}^5$ adapted to this embedding. Then
\begin{equation}\label{eqn:Clifford mult Delta dim 4 and dim 5}
\begin{split}
e_k\cdot_5\phi&=e_k\cdot_4\phi\\
e_0\cdot_5\phi&=ie_1e_2e_3e_4\cdot_4\phi
\end{split}    
\end{equation}
for $k=1,2,3,4$ and all $\phi\in\Delta$. This corresponds to the irreducible representation of $\mathbb{C}\mathrm{l}(5)$ where $\omega_{\mathbb{C}}=-ie_0e_1e_2e_3e_4\in \mathbb{C}\mathrm{l}(5)$ acts as $+1$ (in the other representation this element acts as $-1$, see \cite[Ch.~I, Proposition 5.9]{LM}).

Similarly, the Lie group $\mathrm{Spin}^c(n)$ is defined by
\begin{equation*}
\mathrm{Spin}^c(n)=(\mathrm{Spin}(n)\times \mathrm{U}(1))/\mathbb{Z}_2,
\end{equation*}
where the non-trivial element of $\mathbb{Z}_2$ acts as $(-1,-1)$. There are Lie group homomorphisms
\begin{align*}
\lambda^c\colon\mathrm{Spin}^c(n)&\longrightarrow  \mathrm{SO}(n),\quad[g,z]\longmapsto \lambda(g)\\
\chi\colon\mathrm{Spin}^c(n)&\longrightarrow \mathrm{U}(1),\quad[g,z]\longmapsto z^2.
\end{align*}
These homomorphisms together yield a double covering
\begin{equation*}
(\lambda^c,\chi)\colon \mathrm{Spin}^c(n)\longrightarrow \mathrm{SO}(n)\times\mathrm{U}(1).
\end{equation*}
A $\mathrm{Spin}^c$-structure on $M$ is a principal bundle
\begin{equation*}
P_{\mathrm{Spin}^c}(M)\times\mathrm{Spin}^c(n)\longrightarrow P_{\mathrm{Spin}^c}(M)
\end{equation*}
over $M$ together with a smooth bundle map $\mathfrak{s}^c\colon P_{\mathrm{Spin}^c}(M)\rightarrow P_{\mathrm{SO}}(M)$, so that the following diagram commutes:
\begin{equation*}
\begin{tikzcd}
P_{\mathrm{Spin}^c}(M)&[-3em]\times\ar[dd, "\mathfrak{s}^c\times\lambda^c"]&[-3em] \mathrm{Spin}^c(n)\ar[r] & P_{\mathrm{Spin}^c}(M)\ar[dd, "\mathfrak{s}^c"]\ar[rd] &\\
&& &  &M \\
P_{\mathrm{SO}}(M)&\times &\mathrm{SO}(n)\ar[r] & P_{\mathrm{SO}}(M)\ar[ru] &
\end{tikzcd}
\end{equation*}
For a $\mathrm{Spin}^c$-structure $\mathfrak{s}^c$ the homomorphism $\chi$ defines an associated principal $\mathrm{U}(1)$-bundle
\begin{equation*}
P_{\mathrm{U}(1)}(M)=P_{\mathrm{Spin}^c}(M)\times_\chi \mathrm{U}(1)
\end{equation*}
and with the standard representation of $\mathrm{U}(1)$ on $\mathbb{C}$ a complex line bundle
\begin{equation*}
L=L_{\mathfrak{s}^c}=P_{\mathrm{Spin}^c}(M)\times_\chi \mathbb{C}
\end{equation*} 
with a Hermitian metric, called the {\em characteristic line bundle} of the $\mathrm{Spin}^c$-structure (in the case of $\dim M=4$, the complex line bundle $L$ is sometimes called the determinant line bundle because it is isomorphic to the top exterior power of both Weyl spinor bundles). 

In the following we denote the projection of these principal bundles to $M$ by $\pi_{\mathrm{SO}}$, $\pi_{\mathrm{U}(1)}$, etc. It is sometimes useful to consider the $\mathrm{SO}(n)\times \mathrm{U}(1)$-principal bundle given by the fibre product
\begin{equation*}
P_{\mathrm{SO}\times\mathrm{U}(1)}(M)=\{(p,q)\in P_{\mathrm{SO}}(M)\times P_{\mathrm{U}(1)}(M)\mid \pi_{\mathrm{SO}}(p)=\pi_{\mathrm{U}(1)}(q)\}.
\end{equation*}
A $\mathrm{Spin}^c$-structure $\mathfrak{s}^c$ with associated bundle $P_{\mathrm{U}(1)}(M)$ yields a double covering
\begin{equation*}
\tilde{\mathfrak{s}}^c\colon P_{\mathrm{Spin}^c}(M)\longrightarrow P_{\mathrm{SO}\times\mathrm{U}(1)}(M)
\end{equation*} 
which restricts on each fibre to a double covering equivalent to $(\lambda^c,\chi)$ (see e.g.~\cite[Appendix D]{LM}).

The spinor representation $\kappa_n$ of $\mathrm{Spin}(n)$ together with the standard representation of $\mathrm{U}(1)$ define the unitary (Dirac) spinor representation $\kappa_n^c$ of $\mathrm{Spin}^c(n)$ on the complex vector space $\Delta_n$. A $\mathrm{Spin}^c$-structure defines an associated Hermitian (Dirac) spinor bundle
\begin{equation*}
S_M^c=P_{\mathrm{Spin}^c}(M)\times_{\kappa_n^c}\Delta_n    
\end{equation*}
with a Clifford multiplication. The Levi--Civita connection of the Riemannian metric $g_M$ together with the choice of a Hermitian connection $A$ on $L$ define a Clifford connection $\nabla^M_A$ on $S_M^c$. We get a Dirac operator 
\begin{equation*}
D_A^M\colon \Gamma(S_M^c)\longrightarrow \Gamma(S_M^c).
\end{equation*}
In even dimension $n=\dim M$ there is a splitting $S_M^c=S_M^{c+}\oplus S_M^{c-}$ into Weyl spinor bundles so that the properties corresponding to \eqref{eqn:Clifford Dirac Weyl split} hold.

\begin{defn} Let $\mathcal{S}(M)$ and $\mathcal{S}^c(M)$ denote the set of all isomorphism classes of spin structures and $\mathrm{Spin}^c$-structures on $M$. We denote the isomorphism class of a spin structure $\mathfrak{s}$ or a $\mathrm{Spin}^c$-structure $\mathfrak{s}^c$ by the same symbol.
\end{defn}
The structure of the sets $\mathcal{S}(M)$ and $\mathcal{S}^c(M)$ is well-known.
\begin{prop}
The set $\mathcal{S}(M)$ is non-empty if and only if the second Stiefel-Whitney class of $TM$ vanishes, $w_2(M)=0$. In this case $\mathcal{S}(M)$ is a torsor (a set with a free and transitive action) over $H^1(M;\mathbb{Z}_2)$.
\end{prop}
\begin{prop}\label{prop:spinc torsor}
The set $\mathcal{S}^c(M)$ is non-empty if and only if there exists an element $\alpha\in H^2(M;\mathbb{Z})$ with $w_2(M)\equiv \alpha\bmod 2$. In this case $\mathcal{S}^c(M)$ is a torsor over $H^2(M;\mathbb{Z})$ and 
\begin{equation*}
w_2(M)\equiv c_1(L_{\mathfrak{s}^c})\bmod 2
\end{equation*}
for every $\mathrm{Spin}^c$-structure $\mathfrak{s}^c$. Denoting the action of an element $\alpha\in H^2(X;\mathbb{Z})$ on a $\mathrm{Spin}^c$-structure by $\mathfrak{s}^c\mapsto \mathfrak{s}^c+\alpha$, the action on the associated Chern class is given by $c_1(L_{\mathfrak{s}^c})\mapsto c_1(L_{\mathfrak{s}^c})+2\alpha$. Moreover, for any oriented $4$-manifold $M$ the set $\mathcal{S}^c(M)$ is non-empty.
\end{prop}
\begin{rem}\label{rem:spinc torsor}
We can write the class $\alpha\in H^2(M;\mathbb{Z})$ as $\alpha=c_1(E)$ for a complex line bundle $E$ over $M$ and then denote the action by $\mathfrak{s}^c\mapsto \mathfrak{s}^c\otimes E$ and $L_{\mathfrak{s}^c}\mapsto L_{\mathfrak{s}^c}\otimes E^{\otimes 2}$.
\end{rem}
\begin{defn}
If the sets $\mathcal{S}(M)$ or $\mathcal{S}^c(M)$ are non-empty, we say that $M$ is a {\em spin} or $\mathit{Spin^c}${\em-manifold}, respectively. 
\end{defn}
Hence our notion of spin and $\mathrm{Spin}^c$-manifold signifies only the existence of such a structure and not that a specific one has been chosen.

\subsection{Relation between spin and $\mathrm{Spin}^c$-structures}
Our goal in this section is to understand the correspondence between spin structures and $\mathrm{Spin}^c$-structures with vanishing first Chern class.
\begin{defn}
We define the set
\begin{equation*}
\mathcal{S}^c_0(M)=\{\mathfrak{s}^c\in\mathcal{S}^c(M)\mid c_1(L_{\mathfrak{s}^c})=0\}.
\end{equation*}
\end{defn}
Consider the Bockstein sequence
\begin{equation*}
H^1(M;\mathbb{Z})\stackrel{\cdot 2}{\longrightarrow}H^1(M;\mathbb{Z})\stackrel{\bmod 2}{\longrightarrow}H^1(M;\mathbb{Z}_2)\stackrel{\beta}{\longrightarrow} H^2(M;\mathbb{Z})\stackrel{\cdot 2}{\longrightarrow}H^2(M;\mathbb{Z}).
\end{equation*}
associated to the short exact sequence
\begin{equation*}
0\longrightarrow\mathbb{Z}\stackrel{\cdot 2}{\longrightarrow}\mathbb{Z}\stackrel{\bmod 2}{\longrightarrow}\mathbb{Z}_2\longrightarrow 0.
\end{equation*}
We have
\begin{align*}
\mathrm{ker}\,\beta&\cong H^1(M;\mathbb{Z})/2H^1(M;\mathbb{Z})\cong(\mathbb{Z}_2)^{b_1(M)}\\
\mathrm{im}\,\beta&=\mathrm{Tor}_2(H^2(M;\mathbb{Z})),
\end{align*}
where $\mathrm{Tor}_2$ denotes the elements of order $2$. The following lemma follows immediately from Proposition \ref{prop:spinc torsor}.
\begin{lem}
If the set $\mathcal{S}^c_0(M)$ is non-empty, then it is a torsor over the group $\mathrm{im}\,\beta$.
\end{lem}
The following theorem shows that spin structures induce $\mathrm{Spin}^c$-structures. The first part of this theorem was proved in \cite[p.~49-50]{G}. The second part on isomorphisms is a direct consequence. 
\begin{thm} Let $M$ be a spin manifold. Then there exists a surjective map
\begin{equation*}
\Theta\colon \mathcal{S}(M)\longrightarrow \mathcal{S}^c_0(M)
\end{equation*}
which is equivariant with respect to the Bockstein epimorphism $\beta\colon H^1(M;\mathbb{Z}_2)\rightarrow \mathrm{im}\,\beta$. It descends to an isomorphism of torsors
\begin{equation*}
\bar{\Theta}\colon \mathcal{S}(M)/\mathrm{ker}\,\beta\stackrel{\cong}{\longrightarrow} \mathcal{S}^c_0(M)
\end{equation*}
with respect to the induced group isomorphism $H^1(M;\mathbb{Z}_2)/\mathrm{ker}\,\beta\cong\mathrm{im}\,\beta$.
\end{thm}
In particular, $\mathcal{S}(M)$ decomposes into $|\mathrm{im}\,\beta|$ subsets, each of which contains $|\mathrm{ker}\,\beta|$ isomorphism classes of spin structures that map to the same isomorphism class of $\mathrm{Spin}^c$-structures.
\begin{defn}
We denote by $\mathfrak{s}_0^c=\Theta(\mathfrak{s})$ the $\mathrm{Spin}^c$-structure induced by the spin structure $\mathfrak{s}$. 
\end{defn}
\begin{defn}
Let $\mathfrak{s}$ be a spin structure and $\mathfrak{s}^c$ a $\mathrm{Spin}^c$-structure on $M$ with characteristic line bundle $L$. By Proposition \ref{prop:spinc torsor} and Remark \ref{rem:spinc torsor} there exists a unique complex line bundle $L^{\scriptscriptstyle\frac{1}{2}}$ on $M$, up to isomorphism, so that $\mathfrak{s}^c=\mathfrak{s}_0^c\otimes L^{\scriptscriptstyle\frac{1}{2}}$. This line bundle satisfies $(L^{\scriptscriptstyle\frac{1}{2}})^{\otimes 2}\cong L$ and is called the {\em square root} of $L$ determined by $\mathfrak{s}$.
\end{defn}
A right inverse to the map $\Theta$ can be constructed as follows.
\begin{prop}\label{prop:right inverse Theta}
A $\mathrm{Spin}^c$-structure $\mathfrak{s}^c$ on $M$ with $c_1(L_{\mathfrak{s}^c})=0$ together with a trivialization of $L_{\mathfrak{s}^c}$ determines a unique isomorphism class $\mathfrak{s}$ of spin structures on $M$ so that $\Theta(\mathfrak{s})=\mathfrak{s}^c$.
\end{prop}
\begin{proof}
We first give an explicit construction of the map $\Theta$, cf.~\cite[Example D.5]{LM}. Given a spin structure $\mathfrak{s}\colon P_{\mathrm{Spin}}(M)\rightarrow P_{\mathrm{SO}}(M)$ we define the $\mathrm{Spin}^c(n)$-principal bundle
\begin{equation*}
(P_{\mathrm{Spin}}(M)\times \mathrm{U}(1))/\mathbb{Z}_2,
\end{equation*}
where the non-trivial element of $\mathbb{Z}_2$ acts as $(-1,-1)$. We then get a $\mathrm{Spin}^c$-structure
\begin{equation*}
\Theta(\mathfrak{s})\colon (P_{\mathrm{Spin}}(M)\times \mathrm{U}(1))/\mathbb{Z}_2\longrightarrow P_{\mathrm{SO}}(M),\quad [p,z]\longmapsto \mathfrak{s}(p).
\end{equation*}
Conversely, consider a $\mathrm{Spin}^c$-structure $\mathfrak{s}^c$ with trivialized bundle
\begin{equation*}
P_{\mathrm{U}(1)}(M)=M\times \mathrm{U}(1).
\end{equation*}
There exists a bundle isomorphism
\begin{align*}
P_{\mathrm{SO}}(M)\times \mathrm{U}(1)&\stackrel{\cong}{\longrightarrow} P_{\mathrm{SO}\times\mathrm{U}(1)}(M)\\
(p,z)&\longmapsto (p,(\pi_{\mathrm{SO}}(p),z)).
\end{align*}
The $\mathrm{Spin}^c$-structure defines a double covering 
\begin{equation*}
\tilde{\mathfrak{s}}^c\colon P_{\mathrm{Spin}^c}(M)\longrightarrow P_{\mathrm{SO}}(M)\times \mathrm{U}(1)
\end{equation*}
which is equivariant with respect to the homomorphism $(\lambda^c,\chi)$. The double covering $\tilde{\mathfrak{s}}^c$ restricted to the preimage $P_{\mathrm{Spin}}(M)$ of $P_{\mathrm{SO}}(M)\times\{1\}$ together with the action of $\mathrm{ker}\,\chi\cong\mathrm{Spin}(n)$ defines a spin structure $\mathfrak{s}$ on $M$. By construction $P_{\mathrm{Spin}}(M)\subset P_{\mathrm{Spin}^c}(M)$. The bundle isomorphism
\begin{equation*}
(P_{\mathrm{Spin}}(M)\times \mathrm{U}(1))/\mathbb{Z}_2\longrightarrow P_{\mathrm{Spin}^c}(M),\quad [p,z]\longmapsto pz
\end{equation*}
shows that the spin structure $\mathfrak{s}$ satisfies $\Theta(\mathfrak{s})=\mathfrak{s}^c$.
\end{proof}
\begin{rem}
The construction of the spin structure from the $\mathrm{Spin}^c$-structure is from \cite[Lemma 3.1]{Mor}. In this reference, however, the dependence on the trivialization of $L_{\mathfrak{s}^c}$ is not stated explicitly.
\end{rem}

\section{Spin and $\mathrm{Spin}^c$-structures on Kaluza--Klein circle bundles}\label{sect:spin spinc kk bundles}

\begin{lem}\label{lem:circle bundle spin}
Let $X$ be an oriented manifold and $\pi\colon Y\rightarrow X$ the principal circle bundle with Euler class $e\in H^2(X;\mathbb{Z})$. Then $Y$ is spin if and only if one of the following conditions holds:
\begin{enumerate}
\item $X$ is spin
\item $w_2(X)\equiv e\bmod 2$
\end{enumerate}
\end{lem}
\begin{proof}
We follow the proof in \cite{MHDiss}. The manifold $Y$ is spin if and only if $w_2(Y)=0$. Applying the Whitney sum formula to the decomposition $TY=\pi^*TX\oplus\underline{\mathbb{R}}$, where $\underline{\mathbb{R}}$ is the trivial vertical line bundle, shows that $w_2(Y)=\pi^*w_2(X)$. Similar to the proof of Lemma \ref{lem:only bundles trivial lift Ln} exactness of the long exact $\mathbb{Z}_2$-Gysin sequence
\begin{equation*}
H^0(X;\mathbb{Z}_2)\stackrel{\cup e \bmod 2}{\longrightarrow}H^2(X;\mathbb{Z}_2)\stackrel{\pi^*}{\longrightarrow} H^2(Y;\mathbb{Z}_2)\longrightarrow\ldots
\end{equation*}
together with $H^0(X;\mathbb{Z}_2)\cong\mathbb{Z}_2$ by Convention \ref{convention} imply that the kernel of $\pi^*$ on $H^2(X;\mathbb{Z}_2)$ is equal to $\{0,e\bmod 2\}$.
\end{proof}
Spin structures on the total space of a principal circle bundle can be either {\em even} or {\em odd} with respect to the circle action \cite{AH, Bor, Amm, Sati}. Let $(X,g_X)$ be an oriented Riemannian $n$-manifold and $Y\rightarrow X$ the Kaluza--Klein circle bundle with Riemannian metric $g_Y^A$ and free, isometric $S^1$-action denoted by
\begin{equation*}
\Phi\colon S^1\times Y\longrightarrow Y,\quad (e^{it},y)\mapsto e^{it}y.
\end{equation*}
The induced free action of $S^1$ on $P_{\mathrm{SO}}(Y)$, given by the differential, commutes with the $\mathrm{SO}(n+1)$-action and defines a smooth family of diffeomorphisms
\begin{equation*}
\alpha\colon [0,2\pi]\times P_{\mathrm{SO}}(Y) \longrightarrow P_{\mathrm{SO}}(Y).
\end{equation*}
Suppose that $Y$ is spin and $\mathfrak{s}_Y\colon P_{\mathrm{Spin}}(Y)\rightarrow P_{\mathrm{SO}}(Y)$ a spin structure on $Y$. The family $\alpha$ then lifts to a smooth family of diffeomorphisms
\begin{equation*}
\tilde{\alpha}\colon [0,2\pi] \times P_{\mathrm{Spin}}(Y)\longrightarrow P_{\mathrm{Spin}}(Y)
\end{equation*}
because the manifold $P_{\mathrm{Spin}}(Y)$ is a double covering of $P_{\mathrm{SO}}(Y)$. We have $\alpha_0=\mathrm{Id}$ and specify the lift $\tilde{\alpha}$ uniquely by $\tilde{\alpha}_0=\mathrm{Id}$. Since $\alpha_{2\pi}=\mathrm{Id}$ and the bundles are connected manifolds by Convention \ref{convention}, there are two possibilities:
\begin{itemize}
\item $\tilde{\alpha}_{2\pi}=\mathrm{Id}$: the $S^1$-action on $Y$ then lifts to a free $S^1$-action $S^1\times P_{\mathrm{Spin}}(Y)\rightarrow P_{\mathrm{Spin}}(Y)$.
The $S^1$-action commutes with the $\mathrm{Spin}(n+1)$-action and the projection $P_{\mathrm{Spin}}(Y)\rightarrow P_{\mathrm{SO}}(Y)$ is $S^1$-equivariant. In this case the spin structure $\mathfrak{s}_Y$ is called {\em even} (or {\em projectable}).
\item $\tilde{\alpha}_{2\pi}$ is multiplication with $-1\in\mathrm{Spin}(n+1)$ in each fibre: the $S^1$-action on $Y$ then does not lift to an $S^1$-action on $P_{\mathrm{Spin}}(Y)$. In this case the spin structure $\mathfrak{s}_Y$ is called {\em odd} (or {\em non-projectable}).
\end{itemize}
\begin{rem}\label{rem:even odd spin Spin+1xS1, Spincn+1}
In the case of an even spin structure on $Y$ we get a free action of $\mathrm{Spin}(n+1)\times S^1$ on $P_{\mathrm{Spin}}(Y)$. In the case of an odd spin structure on $Y$ the (non-free) action 
\begin{equation*}
\hat{\Phi}\colon S^1\times Y\longrightarrow Y,\quad (e^{it},y)\longmapsto e^{2it}y
\end{equation*}
does lift to an action on $P_{\mathrm{Spin}}(Y)$, which induces a free action of $\mathrm{Spin}^c(n+1)=(\mathrm{Spin}(n+1)\times S^1)/\mathbb{Z}_2$ on $P_{\mathrm{Spin}}(Y)$. 

The fibres of the principal bundle $P_{\mathrm{Spin}}(Y)\rightarrow Y\rightarrow X$ are principal $\mathrm{Spin}(n+1)$-bundles over $S^1$. These fibres are diffeomorphic to $\mathrm{Spin}(n+1)\times S^1$ in the even case (with sections of the $\mathrm{Spin}(n+1)$-bundles over $S^1$ defined by the lift $\tilde{\alpha}$) and to $\mathrm{Spin}^c(n+1)$ in the odd case.
\end{rem}
\begin{defn}
For a Kaluza--Klein circle bundle $Y\rightarrow X$ we denote the sets of isomorphism classes of even and odd spin structures on $Y$ by $\mathcal{S}_{\mathrm{even}}(Y)$ and $\mathcal{S}_{\mathrm{odd}}(Y)$.
\end{defn}
\begin{defn}
For a class $\delta\in H^2(X;\mathbb{Z})$ define
\begin{equation*}
\mathcal{S}^c_\delta(X)=\{\mathfrak{s}^c\in\mathcal{S}^c(X)\mid c_1(L_{\mathfrak{s}^c})=\delta\}.    
\end{equation*}
\end{defn}
\begin{thm}\label{main thm spin spinc KK}
Let $\pi\colon Y\rightarrow X$ be the Kaluza--Klein circle bundle with Euler class $e\in H^2(X;\mathbb{Z})$ and Riemannian metric $g_Y^A$. Assume that $Y$ is spin. Then there exists a map
\begin{equation*}
\Xi\colon\mathcal{S}(Y)\longrightarrow\mathcal{S}(X)\amalg\mathcal{S}^c_e(X)
\end{equation*}
with the following properties:
\begin{enumerate}
\item The restriction of $\Xi$ to $\mathcal{S}_{\mathrm{even}}(Y)$ is a bijection $\Xi_{\mathrm{even}}\colon\mathcal{S}_{\mathrm{even}}(Y)\stackrel{\cong}{\longrightarrow}\mathcal{S}(X)$.
\item The restriction of $\Xi$ to $\mathcal{S}_{\mathrm{odd}}(Y)$ is a surjection $\Xi_{\mathrm{odd}}\colon\mathcal{S}_{\mathrm{odd}}(Y)\longtwoheadrightarrow\mathcal{S}_e^c(X)$.
\end{enumerate}
\end{thm}
\begin{proof}
For a proof in the case of even spin structures on $Y$ see \cite[p.~236--237]{AB} and \cite[Proposition 2.2]{Bor}. For the construction of $\Xi_{\mathrm{odd}}$ see \cite[p.~242]{AB} and \cite[Proposition 2.3]{Bor}. For both cases, see also \cite[p.~24, 27--29]{Sati}, which refers to \cite{AB} and \cite{Mor}.

In \cite[Proposition 3.1]{Mor} it is shown that every element of $\mathcal{S}_e^c(X)$ lifts to an element of $\mathcal{S}(Y)$. However, it is not immediately clear that this right inverse to $\Xi_{\mathrm{odd}}$ can be assumed to have image in the odd spin structures. We therefore prove this fact (compare with Proposition \ref{prop:right inverse Theta}). Let $\mathfrak{s}_X^c$ be a $\mathrm{Spin}^c$-structure on $X$ with $c_1(L_{\mathfrak{s}_X^c})=e$. We first construct a spin structure on the $\mathrm{SO}(n)$-principal bundle $\pi^*P_{\mathrm{SO}}(X)$ over $Y$: The double covering
\begin{equation*}
\tilde{\mathfrak{s}}_X^c\colon P_{\mathrm{Spin}^c}(X)\longrightarrow P_{\mathrm{SO}\times \mathrm{U}(1)}(X)    
\end{equation*}
is equivariant with respect to the homomorphism $(\lambda^c,\chi)$. We restrict the $\mathrm{Spin}^c(n)$-action on $P_{\mathrm{Spin}^c}(X)$ to a free action of $\mathrm{Spin}(n)\cong\mathrm{ker}\,\chi$. The double covering $\tilde{\mathfrak{s}}_X^c$ is then equivariant with respect to
\begin{equation}\label{eqn:lambdac,1}
\left.(\lambda^c,1)\right|_{\mathrm{ker}\,\chi}\colon\mathrm{Spin}(n)\longrightarrow \mathrm{SO}(n)\times \{1\}\cong\mathrm{SO}(n).    
\end{equation}
We consider the pullback bundles
\begin{align*}
\pi^*P_{\mathrm{Spin}^c}(X)&=\{(y,q)\in Y\times P_{\mathrm{Spin}^c}(X)\mid \pi(y)=\pi_{\mathrm{Spin}^c}(q)\}\\
\pi^*P_{\mathrm{SO}\times\mathrm{U}(1)}(X)&=\{(y,p,w)\in Y\times P_{\mathrm{SO}}(X)\times P_{\mathrm{U}(1)}(X)\mid \pi(y)=\pi_{\mathrm{SO}}(p)=\pi_{\mathrm{U}(1)}(w)\}.
\end{align*}
There is an induced, equivariant double covering
\begin{equation*}
\tilde{\mathfrak{t}}\colon \pi^*P_{\mathrm{Spin}^c}(X)\longrightarrow \pi^*P_{\mathrm{SO}\times\mathrm{U}(1)}(X),\quad (y,q)\longmapsto (y,\tilde{\mathfrak{s}}_X^c(q)). 
\end{equation*}
The assumption on the first Chern class of $L_{\mathfrak{s}_X^c}$ implies that $P_{\mathrm{U}(1)}(X)\cong Y$, hence $\pi^*P_{\mathrm{U}(1)}(X)$ is tautologically trivial and there is an isomorphism
\begin{equation*}
\pi^*P_{\mathrm{SO}}(X)\times \mathrm{U}(1)\stackrel{\cong}{\longrightarrow}\pi^*P_{\mathrm{SO}\times\mathrm{U}(1)}(X),\quad (y,p,z)\longmapsto (y,p,zy).    
\end{equation*}
The restriction of $\tilde{\mathfrak{t}}$ to the preimage $P_{\mathrm{Spin}(n)}(Y)$ of $\pi^*P_{\mathrm{SO}}(X)\times \{1\}$ is equivariant with respect to the homomorphism \eqref{eqn:lambdac,1}, i.e.~a spin structure on $\pi^*P_{\mathrm{SO}}(X)$. We then get a spin structure $\mathfrak{s}_Y$ on $Y$ by extending the fibres via
\begin{equation*}
P_{\mathrm{Spin}}(Y)=P_{\mathrm{Spin}(n)}(Y)\times_j\mathrm{Spin}(n+1),\quad
P_{\mathrm{SO}}(Y)\cong (\pi^*P_{\mathrm{SO}}(X))\times_i\mathrm{SO}(n+1)
\end{equation*}
with respect to compatible embeddings
\begin{equation*}
\begin{tikzcd}
\mathrm{Spin}(n)\arrow[hook]{r}{j}\arrow{d}[swap]{\lambda} & \mathrm{Spin}(n+1)\arrow{d}{\lambda}\\
\mathrm{SO}(n)\arrow[hook]{r}[swap]{i} & \mathrm{SO}(n+1)
\end{tikzcd}
\end{equation*}
It remains to show that $\mathfrak{s}_Y$ is odd. The preimage of $x\in X$ under the composition $P_{\mathrm{Spin}(n)}(Y)\rightarrow Y\rightarrow X$ is
\begin{equation*}
\{(y,q)\in \pi^{-1}(x)\times\pi_{\mathrm{Spin}^c}^{-1}(x)\mid \mathrm{pr}_2\circ\tilde{\mathfrak{s}}_X^c(q)=y\}    
\end{equation*}
where $Y=P_{\mathrm{U}(1)}(X)$. Hence we can identify this preimage with
\begin{equation*}
\{(\alpha,g)\in S^1\times\mathrm{Spin}^c(n)\mid \chi(g)=\alpha\}=\mathrm{graph}(\chi)\cong\mathrm{Spin}^c(n).     
\end{equation*}
It follows that the preimage of $x$ under $P_{\mathrm{Spin}}(Y)\rightarrow Y\rightarrow X$ is diffeomorphic to
\begin{equation*}
\mathrm{Spin}^c(n)\times_j\mathrm{Spin}(n+1)\cong\mathrm{Spin}^c(n+1).
\end{equation*}
This implies the claim by Remark \ref{rem:even odd spin Spin+1xS1, Spincn+1}.
\end{proof}
The following corollary will be used in Definition \ref{defn:main construction Y from X spinc} to fix the choice of Kaluza--Klein circle bundles and construct the lift of spinors.
\begin{cor}\label{main cor even odd spin spinc KK bundles}
Let $(X,g_X)$ be an oriented Riemannian manifold.
\begin{enumerate}
\item If $X$ is $\mathrm{Spin}^c$, let $\mathfrak{s}^c_X$ be a $\mathrm{Spin}^c$-structure on $X$ with characteristic line bundle $L$ and $Y\rightarrow X$ the Kaluza--Klein circle bundle with Euler class $e=c_1(L)$ and Riemannian metric $g_Y^A$. Then there exists an odd spin structure $\mathfrak{s}_Y$ on $Y$ such that $\mathfrak{s}^c_X=\Xi(\mathfrak{s}_Y)$.
\item If $X$ is spin, let $\mathfrak{s}_X$ be a spin structure on $X$ and $Y\rightarrow X$ the Kaluza--Klein circle bundle with arbitrary Euler class $e$ and Riemannian metric $g_Y^A$. Then there exists an even spin structure $\mathfrak{s}_Y$ on $Y$ such that $\mathfrak{s}_X=\Xi(\mathfrak{s}_Y)$.
\end{enumerate}
\end{cor}
\begin{proof}
In both situations the manifold $Y$ is spin by Lemma \ref{lem:circle bundle spin}, because $e=c_1(L)\equiv w_2(X)\bmod 2$ in the first case and $w_2(X)=0$ in the second case. The claims are thus direct consequences of the surjectivity of the map $\Xi$ in Theorem \ref{main thm spin spinc KK}.
\end{proof}

\section{Spinors on Kaluza--Klein circle bundles over $4$-manifolds}\label{sect:spinors KK}

Let $(X,g_X)$ be a closed, oriented, Riemannian $4$-manifold. Our goal is to lift spinors for a given $\mathrm{Spin}^c$-structure on $X$ to spinors for a spin structure on a suitable Kaluza--Klein circle bundle $Y$. It turns out that we can lift sections for a $\mathbb{Z}$-family of $\mathrm{Spin}^c$-spinor bundles on $X$ to the same spinor bundle on $Y$. Let $Y^5\rightarrow X^4$ be the Kaluza--Klein circle bundle with Euler class $e$ and Riemannian metric $g_Y^A$ and $E\rightarrow X$ the complex line bundle with $c_1(E)=e$, unique up to isomorphism. According to \cite{AB}, \cite{Amm}, \cite{AmmDiss} there are two cases:
\begin{enumerate}
\item Let $\mathfrak{s}_Y$ be an odd spin structure on $Y$ and $\mathfrak{s}_X^c=\Xi(\mathfrak{s}_Y)$ with characteristic line bundle $E$. Let $S_Y$ and $S_X^c$ denote the corresponding spinor bundles. Then sections of the spinor bundles $S_X^c\otimes E^{\otimes n}$ can be lifted to sections of $S_Y$ for all $n\in\mathbb{Z}$.
\item Let $\mathfrak{s}_Y$ be an even spin structure on $Y$ and $\mathfrak{s}_X=\Xi(\mathfrak{s}_Y)$. Consider the associated $\mathrm{Spin}^c$-structure $\mathfrak{s}^c_{X0}$ and denote the corresponding spinor bundles by $S_Y$ and $S_{X0}^c$. Then sections of the spinor bundles $S_{X0}^c\otimes E^{\otimes n}$ can be lifted to sections of $S_Y$ for all $n\in\mathbb{Z}$.
\end{enumerate}
Together with Corollary \ref{main cor even odd spin spinc KK bundles} this leads to the following construction:
\begin{defn}[Choice of Kaluza--Klein circle bundles]\label{defn:main construction Y from X spinc} 
Let $(X^4,g_X)$ be a closed, oriented, Riemannian $4$-manifold and $\mathfrak{s}_X^c$ a $\mathrm{Spin}^c$-structure on $X$ with characteristic line bundle $L$ and spinor bundle $S_X^c$. We consider two choices for the Kaluza--Klein circle bundle $Y^5\rightarrow X^4$ with Riemannian metric $g_Y=g_Y^A=\pi^*g_X-A\otimes A$, depending on whether the spin structure on $Y$ should be even or odd (see Table \ref{table:choice KK circle bundle}):
\begin{enumerate}
\item {\bf Odd spin structure} on $Y$: For this case the manifold $X$ can be spin or non-spin. Let $\pi\colon Y\rightarrow X$ be the Kaluza--Klein circle bundle with Euler class $e(Y)=c_1(L)$, $A$ a Hermitian connection on $L$ and $\mathfrak{s}_Y$ an odd spin structure on $Y$ with spinor bundle $S_Y$ such that $\mathfrak{s}_X^c=\Xi(\mathfrak{s}_Y)$.
\item {\bf Even spin structure} on $Y$: For this case the manifold $X$ has to be spin. Let $\mathfrak{s}_X$ be one of the spin structures on $X$ and $L^{\scriptscriptstyle\frac{1}{2}}$ the square root of $L$ determined by $\mathfrak{s}_X$. Let $\pi\colon Y\rightarrow X$ be the Kaluza--Klein circle bundle with Euler class $e(Y)=c_1(L^{\scriptscriptstyle\frac{1}{2}})$, $A$ a Hermitian connection on $L^{\scriptscriptstyle\frac{1}{2}}$ and $\mathfrak{s}_Y$ an even spin structure on $Y$  with spinor bundle $S_Y$ such that $\mathfrak{s}_X=\Xi(\mathfrak{s}_Y)$.
\end{enumerate}
In both cases we call the spin structure $\mathfrak{s}_Y$ on $Y$ a {\em lift} of the $\mathrm{Spin}^c$-structure $\mathfrak{s}_X^c$ (the isomorphism class of the lift $\mathfrak{s}_Y$ is not necessarily uniquely determined by $\mathfrak{s}_X^c$). Sections of the $\mathrm{Spin}^c$-spinor bundles $S_X^{c,n}$, with characteristic line bundles $L_n$, lift to sections of the spinor bundle $S_Y$ (see Remark \ref{rem:precise statement lift spinor} below).
\begin{table}
{\footnotesize
\caption{Kaluza--Klein circle bundles $Y^5\rightarrow X^4$ and spinor bundles $S_X^{c,n}$ of charge $q$}
\label{table:choice KK circle bundle}
\renewcommand{\arraystretch}{1.5}
\setlength{\tabcolsep}{0.75em}
\begin{tabular}{lccccccc} 
\hline\noalign{\smallskip}
$X$ & $e(Y)$ & $A$ & $\mathfrak{s}_Y$ & $S_X^{c,n}$ & $L_n$ & $n$ & $q$\\
\noalign{\smallskip}\hline\noalign{\smallskip}
spin or non-spin & $c_1(L)$ & $L$ & odd &  $S_X^c\otimes L^{\otimes n}$ & $L^{\otimes (1+2n)}$ & $n\in\mathbb{Z}$ & $\tfrac{1}{2}+n$ \\
spin & $c_1(L^{\scriptscriptstyle\frac{1}{2}})$ & $L^{\scriptscriptstyle\frac{1}{2}}$ & even &  $S_X^c\otimes (L^{\scriptscriptstyle\frac{1}{2}})^{\otimes n}$ & $L^{\otimes (1+n)}$ & $n\in\mathbb{Z}\setminus\{-1\}$ & $1+n$ \\
\noalign{\smallskip}\hline
\end{tabular}
}
\end{table}
\end{defn}
\begin{rem}
Suppose that $X$ is spin, $\mathfrak{s}_X$ one of the spin structures and $\mathfrak{s}_X^c$ a $\mathrm{Spin}^c$-structure with characteristic line bundle $L$. Then
\begin{equation*}
S_X^c=S_{X0}^c\otimes L^{\scriptscriptstyle\frac{1}{2}}.
\end{equation*}
This is the reason for choosing the circle bundle $Y\rightarrow X$ with Euler class $e=c_1(L^{\scriptscriptstyle\frac{1}{2}})$ in the second case of Definition \ref{defn:main construction Y from X spinc}.
\end{rem}
\begin{rem}\label{rem:charge q}
The numbers $q\in \frac{1}{2}+\mathbb{Z}$ (if $\mathfrak{s}_Y$ is odd) and $q\in\mathbb{Z}$ (if $\mathfrak{s}_Y$ is even) can be thought of as the $\mathrm{U}(1)$-charges of the spinor bundle $S_X^{c,n}$ if the bundles $L$ and $L^{\scriptscriptstyle\frac{1}{2}}$, respectively, correspond to charge $1$. This follows because $S_X^{c,n}=S_{X0}^c\otimes L^{\otimes q}$ in the first case and $S_X^{c,n}=S_{X0}^c\otimes (L^{\scriptscriptstyle\frac{1}{2}})^{\otimes q}$ in the second case if a spin structure $\mathfrak{s}_X$ exists on $X$ (which is always the case locally) (cf.~\cite{W}).
\end{rem}
\begin{rem}
In the case of an even spin structure on $Y$ we exclude the case $n=-1$, corresponding to $q=0$, for reasons explained in Remark \ref{rem:exclude q=0}.
\end{rem}
\begin{rem}\label{rem:precise statement lift spinor}
A more precise statement concerning the lifts of spinors is the following \cite{AB}, \cite{Amm}, \cite{AmmDiss}: For both the case of an odd and even spin structure $\mathfrak{s}_Y$ on $Y$, the $S^1$-action on $Y$ defines a Lie derivative of the Killing vector field $K$ on sections of the spinor bundle $S_Y$, denoted by
\begin{equation*}
L_K\colon\Gamma(S_Y)\longrightarrow \Gamma(S_Y).
\end{equation*}
The space $L^2(S_Y)$ of $L^2$-sections can be decomposed into eigenspaces of $L_K$: 
\begin{enumerate}
\item In the case of an odd spin structure $\mathfrak{s}_Y$ there is a decomposition
\begin{equation*}
L^2(S_Y)=\bigoplus_{k\in\mathbb{Z}}V_{k+\frac{1}{2}},
\end{equation*}
where $V_{k+\frac{1}{2}}\subset L^2(S_Y)$ is the eigenspace of $L_K$ with eigenvalue $i(k+\frac{1}{2})$.
\item In the case of an even spin structure $\mathfrak{s}_Y$ there is a decomposition
\begin{equation*}
L^2(S_Y)=\bigoplus_{k\in\mathbb{Z}}V_k,
\end{equation*}
where $V_k\subset L^2(S_Y)$ is the eigenspace of $L_K$ for the eigenvalue $ik$.
\end{enumerate}
Note that by our assumption that $\dim X=4$ is even, the bundles $S_Y$ and $S_X^{c,n}$ have the same complex rank $4$. For both cases in Definition \ref{defn:main construction Y from X spinc} there exists a canonical bundle map
\begin{equation*}
\Pi_{-q}\colon S_Y\longrightarrow S_X^{c,n}
\end{equation*}
that covers the projection $\pi\colon Y\rightarrow X$ and is fibrewise an isomorphism. It induces an isomorphism of Hilbert spaces
\begin{equation*}
Q_{-q}\colon L^2(S_X^{c,n})\longrightarrow V_{-q}\subset L^2(S_Y),
\end{equation*}
so that the following diagram commutes for each spinor $\phi\in\Gamma(S_X^{c,n})$:
\begin{equation*}
\begin{tikzcd}
Y\ar[r, "Q_{-q}(\phi)"]\ar[d, "\pi"] & S_Y\ar[d, "\Pi_{-q}"]\\
X\ar[r, "\phi"]& S_X^{c,n}
\end{tikzcd}
\end{equation*}
The map $Q_{-q}$ commutes with Clifford multiplication:
\begin{equation}\label{eqn:Q_-q commutes with Clifford}
Q_{-q}(V\cdot\phi)=V^*\cdot Q_{-q}(\phi)
\end{equation}
for all $\phi\in\Gamma(S_X^{c,n})$ and $V\in \mathfrak{X}(X)$ with horizontal lift $V^*\in\mathfrak{X}(Y)$ with respect to the Kaluza--Klein metric $g_Y^A$.
\end{rem}
\begin{defn}
We call the spinor $Q_{-q}(\phi)\in \Gamma(S_Y)$ the {\em lift} of the spinor $\phi\in\Gamma(S_X^{c,n})$. We denote $Q_{-q}$ often by $Q$.
\end{defn}
\begin{rem}
It follows that $Q_{-q}$ is a correspondence between spinors on $X$ of $\mathrm{U}(1)$-charge $q$ and spinors on $Y$ of eigenvalue $-iq$ under $L_K$.
\end{rem}
\begin{rem}
The decomposition 
\begin{equation*}
S_X^c=S_X^{c+}\oplus S_X^{c-}
\end{equation*}
into positive and negative Weyl spinor bundles extends to all twisted spinor bundles $S_X^{c,n}$. Via $Q_{-q}$ we get corresponding decompositions
\begin{equation*}
V_{-q}=V_{-q}^+\oplus V_{-q}^-
\end{equation*}
of spinors in $\Gamma(S_Y)$.
\end{rem}
\begin{lem}\label{lem:relation Clifford 4 and 5d}
Clifford multiplication on the spinor bundles $S_X^{c,n}$ and $S_Y$ is related by 
\begin{align*}
V^*\cdot Q(\phi)&=Q(V\cdot\phi)\\
K\cdot Q(\phi)&=Q(i\mathrm{dvol}_{g_X}\cdot \phi)\\
\pi^*\kappa\cdot Q(\phi)&=Q(\kappa\cdot\phi),
\end{align*}
for all $\phi\in\Gamma(S_X^{c,n})$, where the vector field $V^*\in \mathfrak{X}(Y)$ is the horizontal lift (with respect to the metric $g_Y$) of a vector field $V\in \mathfrak{X}(X)$, $K$ is the vertical unit Killing vector field along the $S^1$-fibres, $\mathrm{dvol}_{g_X}$ denotes the volume form and $\kappa\in\Omega^*(X)$. In particular, Clifford multiplication with $V^*$ exchanges $V_{-q}^\pm$ and with $K$ preserves $V_{-q}^\pm$.
\end{lem}
This follows from equations \eqref{eqn:Clifford mult Delta dim 4 and dim 5} and \eqref{eqn:Q_-q commutes with Clifford} where the orientation of $TY$ is defined by $K$ followed by the orientation of $\pi^*TX$.
\begin{rem}
Analogous equations can be found in \cite[eqns.~(10), (11)]{Mor}, however, in this reference $K\cdot Q(\phi)=-Q(i\mathrm{dvol}_{g_X}\cdot \phi)$, corresponding to the representation of $\mathbb{C}\mathrm{l}(5)$ where $\omega_{\mathbb{C}}$ acts as $-1$.
\end{rem}
\begin{rem}\label{rem:choice of compatible Hermitian bundle metrics}
We choose the invariant scalar products on $\Delta_4\cong\Delta_5\cong\mathbb{C}^4$ and $\mathbb{C}$ such that the bundle maps $\Pi_{-q}$ in Remark \ref{rem:precise statement lift spinor} are isometries on each fibre with respect to the induced bundle metrics. For spinors $\phi$ and $\psi=Q_{-q}(\phi)$, this implies
\begin{equation*}
|\psi|^2=|\phi|^2\circ\pi.
\end{equation*}
\end{rem}

\section{Dirac operators}
Let $(X^4,g_X)$ be a closed, oriented, Riemannian $4$-manifold with $\mathrm{Spin}^c$-structure $\mathfrak{s}_X^c$ and $\pi\colon Y\rightarrow X$ one of the two Kaluza--Klein circle bundles given by Definition \ref{defn:main construction Y from X spinc}. 

\subsection{Clifford identities}
We collect some identities involving Clifford multiplication. Let $\Omega^2_\pm(X,\mathbb{C})$ denote the self-dual and anti-self-dual $2$-forms on $X$, satisfying 
\begin{equation*}
*\omega_\pm=\pm\omega_\pm,\quad\forall \omega_\pm\in \Omega^2_\pm(X),
\end{equation*}
where $*$ is the Hodge star operator. 
\begin{lem}\label{lem:clifford identities} The following identities for Clifford multiplication $\gamma$ hold:
\begin{enumerate}
\item For $\phi_\pm\in \Gamma(S_X^{n,c\pm})$ we have $\mathrm{dvol}_{g_X}\cdot\phi_\pm=\mp \phi_\pm$.
\item For $\omega_\pm\in\Omega^2_\pm(X,\mathbb{C})$ we have
$\mathrm{dvol}_{g_X}\cdot\omega_\pm=\mp\omega_\pm$ and
\begin{equation*}
\gamma(\omega_\pm) \in \Gamma(\mathrm{End}_{0}(S_X^{n,c\pm})),
\end{equation*}
where $\mathrm{End}_{0}(S_X^{n,c\pm})$ denotes the bundle of trace-free endomorphisms of $S_X^{n,c\pm}$. In particular,
\begin{equation*}
\gamma(\omega_-)\equiv 0\quad\text{on $\Gamma(S_X^{n,c+})$}.
\end{equation*}
\end{enumerate}
\end{lem}
\begin{proof}
These identities can be found in references on Seiberg--Witten theory, e.g.~\cite{Morgan}. The Weyl spinor bundles are defined by $\omega_{\mathbb{C}}$ in equation \eqref{eqn:omegaC dim 4}.
\end{proof}
This implies together with Lemma \ref{lem:relation Clifford 4 and 5d}:
\begin{prop}\label{prop:clifford identities S+} Let $\phi\in \Gamma(S_X^{n,c+})$ be a positive Weyl spinor, $\psi=Q(\phi)$ and $\omega\in \Omega^2(X,\mathbb{C})$ with self-dual part $\omega_+$. Then
\begin{align*}
K\cdot\psi&=-i\psi\\
K\cdot \pi^*\omega\cdot\psi&=-iQ(\omega_+\cdot\phi).
\end{align*}
\end{prop}

\subsection{Relation between Dirac operators on $X$ and $Y$}
Let $A$ be a Hermitian connection on the Hermitian complex line bundle $L$ or $L^{\scriptscriptstyle\frac{1}{2}}$ according to Table \ref{table:choice KK circle bundle}.
\begin{defn} We denote by $A_n=2qA$ the Hermitian connection on the line bundle $L_n$ induced from $A$.
\end{defn}
Together with the Levi--Civita connection of $g_X$ this defines the Dirac operator
\begin{equation*}
D_{A_n}^X\colon \Gamma(S_X^{c,n\pm})\longrightarrow \Gamma(S_X^{c,n\mp}).
\end{equation*}
The connection $A$ and Riemannian metric $g_X$ define the Kaluza--Klein metric $g_Y=g_Y^A$ on $Y$, that yields the Dirac operator
\begin{equation*}
D^Y\colon \Gamma(S_Y)\longrightarrow \Gamma(S_Y).
\end{equation*}
The curvature $F_A\in\Omega^2(X,i\mathbb{R})$ satisfies $\pi^*F_A=dA$. According to Remark \ref{rem:precise statement lift spinor} there is an isometry
\begin{equation*}
Q=Q_{-q}\colon L^2(S_X^{c,n})\longrightarrow V_{-q}\subset L^2(S_Y).
\end{equation*}
\begin{lem}\label{lem:AB calculation Dirac operator}
The covariant derivatives on the spinor bundles $S_X^{c,n}$ and $S_Y$ are related by:
\begin{align*}
\nabla^Y_{V^*}Q(\phi)&=Q(\nabla_{A_nV}^X\phi)-\frac{1}{4}i\gamma(K)\gamma(\pi^*(i_VF_A))Q(\phi)\\
\nabla^Y_KQ(\phi)&=L_KQ(\phi)-\frac{1}{4}i\gamma(\pi^*F_A)Q(\phi),
\end{align*}
where $V\in TX$ with horizontal lift $V^*\in TY$. This implies for the Dirac operator $D^Y$ the formula
\begin{equation*}
D^YQ(\phi)=Q(D_{A_n}^X\phi)+\gamma(K)L_KQ(\phi)+\frac{1}{4}i\gamma(K)\gamma(\pi^*F_A)Q(\phi).
\end{equation*}
\end{lem}
For a proof of these formulas see \cite[Lemma 4.3, Lemma 4.4 and the proof of Theorem 4.1]{AB}. The claim for the Dirac operator follows from the formula 
\begin{equation*}
\sum_{l=1}^4\gamma(e_l)\gamma(i_{e_l}F_A)=2\gamma(F_A),
\end{equation*}
where $\{e_l\}_{l=1}^4$ is a local orthonormal frame on $X$. The formula for the covariant derivative on $S_Y$ in the case of $q=\frac{1}{2}$ has also been proved in \cite[Proposition 3.2]{Mor}.

Recall that spinors $\psi\in V_{-q}$ satisfy $L_K\psi=-qi\psi$. This implies (for the second statement we use Proposition \ref{prop:clifford identities S+}):
\begin{prop}\label{prop:formula Dirac DY spin and non-spin}
The restriction
\begin{equation*}
D^Y\colon V_{-q} \longrightarrow V_{-q}
\end{equation*}
is given by
\begin{equation*}
D^Y=Q\circ D_{A_n}^X\circ Q^{-1}-i\gamma(K)\left(q-\frac{1}{4}\gamma(\pi^*F_A)\right).
\end{equation*}
Using $F_{A_n}=2qF_A$ and setting $m=-q$ the restriction
\begin{equation*}
D^Y\colon V_{-q}^+ \longrightarrow V_{-q}
\end{equation*}
for all $q\neq 0$ is given by
\begin{equation*}
D^Y=Q\circ \left(D_{A_n}^X+m-\frac{1}{8m}\gamma(F_{A_n}^+)\right)\circ Q^{-1}.
\end{equation*}
\end{prop}
\begin{rem}
The number $m$ can be interpreted as the mass of the lifted spinor $\psi=Q(\phi)$, cf.~Remark \ref{rem:GN}.
\end{rem}

\section{Lift of the Seiberg--Witten equations}

\subsection{The Seiberg--Witten equations}
Let $(X^4,g_X)$ be a closed, oriented, Riemannian $4$-manifold with a $\mathrm{Spin}^c$-structure $\mathfrak{s}_X^c$ and characteristic line bundle $L$. We consider in this subsection a spinor $\phi\in\Gamma(S_X^{c+})$ and a Hermitian connection $A$ on $L$. The Seiberg--Witten equations \cite{W} for $(A,\phi)$ are
\begin{align*}
D_A^X\phi&=0\\
F_A^+&=\sigma(\phi,\phi).
\end{align*}
We follow the notation in \cite{Ko}. Here $\sigma(\phi,\phi)$ is the imaginary-valued self-dual $2$-form in $\Omega^2_+(X,i\mathbb{R})$ which under the fibrewise isomorphism
\begin{equation}\label{eqn:isom Omega2+ EndSX+}
\gamma\colon \Lambda^2_+T^*X\otimes\mathbb{C}\longrightarrow \mathrm{End}_{0}(S_X^{c+})
\end{equation} 
corresponds to the trace-free endomorphism
\begin{equation*}
\gamma(\sigma(\phi,\phi))=(\phi\otimes \phi^\dagger)_0=\phi\otimes \phi^\dagger-\frac{1}{2}\mathrm{Tr}(\phi\otimes \phi^\dagger)\mathrm{Id}.
\end{equation*}
Using an explicit representation of the spinor space $\Delta_4$ as a module over the Clifford algebra $\mathbb{C}\mathrm{l}(4)$ it can be shown \cite{Morgan} that $\gamma$ maps the real forms $\Omega^2_+(X,\mathbb{R})$ isomorphically onto the skew-Hermitian trace-free endomorphisms of $S_X^{c+}$. The imaginary-valued forms $\Omega^2_+(X,i\mathbb{R})$ thus map isomorphically onto the Hermitian trace-free endomorphisms of $S_X^{c+}$. Hence $\sigma(\phi,\phi)$ is indeed a form in $\Omega^2_+(X,i\mathbb{R})$. We also get:
\begin{lem}\label{lem:gamma(tau) kernel}
Let $p\in X$, $\tau_p\in i\Lambda^2_+T_p^*X$ and $\phi_p\in S_{Xp}^{c+}$. Then the following holds: 
\begin{equation*}
\gamma(\tau_p)\phi_p=0 \Leftrightarrow \tau_p=0\text{ or }\phi_p=0.
\end{equation*}
\end{lem}
\begin{proof} This follows from the formula
\begin{equation*}
\gamma(\tau_p)\gamma(\tau_p)\phi_p=2|\tau_p|^2\phi_p,
\end{equation*}
that holds for all $\tau_p,\phi_p$ as in the statement of the lemma.
\end{proof}
\begin{lem}\label{lem:sigma(phi,phi)phi}
For $\phi\in \Gamma(S_X^{c+})$ we have 
\begin{equation*}
\gamma(\sigma(\phi,\phi))\phi=\tfrac{1}{2}|\phi|^2\phi.
\end{equation*}
\end{lem}
\begin{proof}
With respect to a local orthonormal frame for $S_X^{c+}$ we can write (see \cite{Ko})
\begin{equation}\label{eqn:Clifford action sigma phi phi}
\phi=\left(\begin{array}{c} \alpha \\ \beta  \end{array}\right)\in \mathbb{C}^2,\quad \gamma(\sigma(\phi,\phi))=\left(\begin{array}{cc} \tfrac{1}{2}(|\alpha|^2-|\beta|^2) & \alpha\bar{\beta} \\ \bar{\alpha}\beta & \tfrac{1}{2}(|\beta|^2-|\alpha|^2)  \end{array}\right),
\end{equation}
hence the action of $\sigma(\phi,\phi)$ on $\phi$ is given by
\begin{align*}
\gamma(\sigma(\phi,\phi))\phi&=\left(\begin{array}{cc} \tfrac{1}{2}(|\alpha|^2-|\beta|^2) & \alpha\bar{\beta} \\ \bar{\alpha}\beta & \tfrac{1}{2}(|\beta|^2-|\alpha|^2)  \end{array}\right)\left(\begin{array}{c} \alpha \\ \beta  \end{array}\right)\\
&=\tfrac{1}{2}(|\alpha|^2+|\beta|^2)\left(\begin{array}{c} \alpha \\ \beta  \end{array}\right).
\end{align*}
\end{proof}
One often considers the more general Seiberg--Witten equations
\begin{align*}
D_A^X\phi&=0\\
F_A^+&=\sigma(\phi,\phi)+\mu,
\end{align*}
where $\mu\in \Omega^2_+(X,i\mathbb{R})$ is an arbitrary perturbation.
\begin{prop}\label{prop:equiv SW multiplied phi}
If $(A,\phi)$ satisfy
\begin{equation}\label{eqn:SW FA omega}
F_A^+=\sigma(\phi,\phi)+\mu,
\end{equation}
then
\begin{equation}\label{eqn:SW FA omega applied phi}
\gamma(F_A^+)\phi=\left(\tfrac{1}{2}|\phi|^2+\gamma(\mu)\right)\phi.
\end{equation}
Conversely suppose that $(A,\phi)$ satisfy equation \eqref{eqn:SW FA omega applied phi} and
\begin{equation*}
D_A^X\phi=0.
\end{equation*}
If $\phi$ does not vanish identically on $X$, then $(A,\phi)$ satisfy equation \eqref{eqn:SW FA omega}.
\end{prop}
\begin{proof}
The first claim is immediate by Lemma \ref{lem:sigma(phi,phi)phi}. For the converse, equation \eqref{eqn:SW FA omega applied phi} implies
\begin{equation*}
\gamma(F_A^+-\sigma(\phi,\phi)-\mu)\phi=0.
\end{equation*}
Lemma \ref{lem:gamma(tau) kernel} implies that equation \eqref{eqn:SW FA omega} holds in all points $p\in X$ where $\phi_p$ is non-zero. Suppose that $F_A^+-\sigma(\phi,\phi)-\mu$ is non-zero in some point $p\in X$. Then by continuity it is non-zero also in a small open neighbourhood $U$ of $p$. Hence $\phi$ has to vanish on $U$. Since $\phi$ is in the kernel of $D_A^X$, the unique continuation property of Dirac operators \cite{Aro, Boo, Baer} and the assumption that $X$ is connected by Convention \ref{convention} imply that $\phi$ vanishes identically on $X$, contradicting the assumption. Therefore equation \eqref{eqn:SW FA omega} holds in all points of $X$.
\end{proof}
Recall that a pair $(A,\phi)$ is called irreducible if the spinor $\phi$ does not vanish identically on $X$. Similarly we define for the Kaluza--Klein circle bundle $Y\rightarrow X$:
\begin{defn}
A spinor $\psi\in\Gamma(S_Y)$ is called {\em irreducible} if 
$\psi$ does not vanish identically on $Y$.
\end{defn}

\subsection{Main theorem}
We can now state our main theorem.
\begin{thm}\label{thm:main thm all n}
Let $(X^4,g_X)$ be a closed, oriented, Riemannian $4$-manifold and $\mathfrak{s}^c_X$ a $\mathrm{Spin}^c$-structure on $X$ with characteristic line bundle $L$. With the notations from Definition \ref{defn:main construction Y from X spinc} let $\pi\colon Y\rightarrow X$ be one of the two principal circle bundles with Euler class $e(Y)$ and $\mathfrak{s}_Y$ a lift of $\mathfrak{s}_X^c$ to a spin structure on $Y$. We define a map
\begin{equation*}
(A_n,\phi,g_X,\mu)\stackrel{\pi^*}{\longmapsto} (\psi,g_Y^A,\pi^*\mu)
\end{equation*}
for a positive Weyl spinor $\phi\in\Gamma(S_X^{c,n+})$, Hermitian connection $A_n$ on $L_n$ and perturbation $\mu\in\Omega^2_+(X,i\mathbb{R})$
as follows:
\begin{itemize}
\item $\psi=Q_{-q}(\phi)\in\Gamma(S_Y)$ is the lift of $\phi$
\item $g_Y^A$ is the Kaluza--Klein metric on $Y$ determined by $g_X$ and the connection $A=\frac{1}{2q}A_n$ on $Y$ according to Definition \ref{defn:KK metric}.
\end{itemize}
Let $m=-q\neq 0$. Then the following holds for $(\psi,g_Y^A,\pi^*\mu)=\pi^*(A_n,\phi,g_X,\mu)$:
\begin{enumerate}
\item If $(A_n,\phi)$ is a solution of the Seiberg--Witten equations
\begin{equation}\label{eqn:main thm SW eqns}
\begin{split}
D_{A_n}^X\phi&=0\\
F_{A_n}^+&=\sigma(\phi,\phi)+\mu
\end{split}
\end{equation}
for parameters $(g_X,\mu)$, then $\psi$ is a solution of the equation
\begin{equation}\label{eqn:DY equivalent SW eqns}
D^Y\psi=-\frac{1}{16m}|\psi|^2\psi+m\psi-\frac{1}{8m}\gamma(\pi^*\mu)\psi
\end{equation}
for parameters $(g_Y^A,\pi^*\mu)$.
\item Conversely, if $\psi$ is an irreducible solution of equation \eqref{eqn:DY equivalent SW eqns} for parameters $(g_Y^A,\pi^*\mu)$, then $(A_n,\phi)$ is an irreducible solution of equations \eqref{eqn:main thm SW eqns} for parameters $(g_X,\mu)$. 
\end{enumerate}
\end{thm}
\begin{proof}
If $(A_n,\phi)$ satisfy the Seiberg--Witten equations \eqref{eqn:main thm SW eqns}, then by Proposition \ref{prop:equiv SW multiplied phi}
\begin{equation*}
\gamma(F_{A_n}^+)\phi=\left(\frac{1}{2}|\phi|^2+\gamma(\mu)\right)\phi.
\end{equation*}
Hence by Proposition \ref{prop:formula Dirac DY spin and non-spin} and Remark \ref{rem:choice of compatible Hermitian bundle metrics}
\begin{equation}\label{eqn:DY proof main thm}
\begin{split}
D^Y\psi&=Q\circ \left(D_{A_n}^X+m-\frac{1}{8m}\gamma(F_{A_n}^+)\right)\phi\\
&=m\psi-\frac{1}{16m}|\psi|^2\psi-\frac{1}{8m}\gamma(\pi^*\mu)\psi.
\end{split}
\end{equation}
The converse follows, because the term in the second line of equation \eqref{eqn:DY proof main thm} is an element of $Q(\Gamma(S_X^{c,n+}))$, hence together with
\begin{equation*}
D_{A_n}^X\phi\in \Gamma(S_X^{c,n-})\quad\text{and}\quad \left(m-\frac{1}{8m}\gamma(F_{A_n}^+)\right)\phi\in\Gamma(S_X^{c,n+})
\end{equation*}
we get
\begin{align*}
D_{A_n}^X\phi&=0\\
\gamma(F_{A_n}^+)\phi&=\left(\frac{1}{2}|\phi|^2+\gamma(\mu)\right)\phi.
\end{align*}
Under the assumption that $\phi\neq 0$ somewhere on $X$, Proposition \ref{prop:equiv SW multiplied phi} implies the Seiberg--Witten equations \eqref{eqn:main thm SW eqns}.
\end{proof}
\begin{rem}
The spinor $\psi\equiv 0$ is always a solution of equation \eqref{eqn:DY equivalent SW eqns}. The corresponding pair $(A_n,0)$ is a solution of equations \eqref{eqn:main thm SW eqns} only if $F_{A_n}^+=\mu$.
\end{rem}
\begin{rem}
Theorem \ref{thm:main thm n=1} is the special case of Theorem \ref{thm:main thm all n} for odd spin structure $\mathfrak{s}_Y$ on $Y$, $q=\frac{1}{2}$ and $\mu=0$.
\end{rem}
\begin{rem}\label{rem:exclude q=0}
Equation \eqref{eqn:DY equivalent SW eqns} does not make sense if $m=-q=0$, hence we exclude the case $n=-1$ for an even spin structure $\mathfrak{s}_Y$ on $Y$.
\end{rem}
\begin{rem}\label{rem:depend sign ident Cliff 4 5 dim}
Equation \eqref{eqn:DY equivalent SW eqns} depends on the relation between Clifford multiplication on $X$ and $Y$, cf.~equation \eqref{eqn:Clifford mult Delta dim 4 and dim 5} and Lemma \ref{lem:relation Clifford 4 and 5d}. If the other choice with $K\cdot Q(\phi)=-Q(i\mathrm{dvol}_{g_X}\cdot \phi)$ is made, then all terms on the right hand side of equation \eqref{eqn:DY equivalent SW eqns} change sign:
\begin{equation*}
D^Y\psi=\frac{1}{16m}|\psi|^2\psi-m\psi+\frac{1}{8m}\gamma(\pi^*\mu)\psi.
\end{equation*}
\end{rem}

\section{Invariance under gauge transformations and charge conjugation}\label{sect:gauge transform charge conj}

\subsection{Gauge transformations}
For a smooth function $h\colon X\rightarrow S^1$ let
\begin{equation*}
d\theta_h=h^{-1}\cdot dh\in\Omega^1(X,i\mathbb{R}).    
\end{equation*}
Consider the gauge transformation $(A_n,\phi)\mapsto (A_n^h,\phi^h)$ given by Table \ref{table:gauge transformations}.
\begin{table}
{\footnotesize
\caption{Gauge transformations defined by $h\colon X\rightarrow S^1$}
\label{table:gauge transformations}
\renewcommand{\arraystretch}{1.5}
\setlength{\tabcolsep}{0.75em}
\begin{tabular}{lccccccc} 
\hline\noalign{\smallskip}
$X$ & $A$ & $S_X^{c,n}$ & $L_n$ & $A^h$ & $A_n^h$ & $\phi^h$ & $q$\\
\noalign{\smallskip}\hline\noalign{\smallskip}
spin or non-spin & $L$ & $S_X^c\otimes L^{\otimes n}$ & $L^{\otimes (1+2n)}$ & $A+2\pi^*d\theta_h$ & $2qA^h$ & $h^{-2q}\phi$ & $\tfrac{1}{2}+n$ \\
spin & $L^{\scriptscriptstyle\frac{1}{2}}$ &  $S_X^c\otimes (L^{\scriptscriptstyle\frac{1}{2}})^{\otimes n}$ & $L^{\otimes (1+n)}$ & $A+\pi^*d\theta_h$ & $2qA^h$ & $h^{-q}\phi$ & $1+n$ \\
\noalign{\smallskip}\hline
\end{tabular}
}
\end{table}
In the situation of Theorem \ref{thm:main thm all n}, if $(A_n,\phi)$ is a solution to the Seiberg--Witten equations \eqref{eqn:main thm SW eqns}, then so is $(A_n^h,\phi^h)$ for every smooth map $h\colon X\rightarrow S^1$.

On the principal circle bundle $\pi\colon Y\rightarrow X$ define 
\begin{align*}
g_Y^h&=\pi^*g_X-A^h\otimes A^h\\
\psi^h&=Q_{-q}(\phi^h)
\end{align*}
and the Dirac operator $D^{Y,h}$ associated to the spin structure $\mathfrak{s}_Y$ and Riemannian metric $g_Y^h$.
We get:
\begin{prop}
With the notation from Theorem \ref{thm:main thm all n} let
\begin{align*}
(\psi,g_Y,\pi^*\mu)&=\pi^*(A_n,\phi,g_X,\mu)\\
(\psi^h,g_Y^h,\pi^*\mu)&=\pi^*(A_n^h,\phi^h,g_X,\mu)
\end{align*}
and $m=-q$. If $\psi$ is a solution to the equation
\begin{equation*}
D^Y\psi=-\frac{1}{16m}|\psi|^2\psi+m\psi-\frac{1}{8m}\gamma(\pi^*\mu)\psi
\end{equation*}
with parameters $(g_Y,\pi^*\mu)$, then $\psi^h$ is a solution to the equation
\begin{equation*}
D^{Y,h}\psi^h=-\frac{1}{16m}|\psi^h|^2\psi^h+m\psi^h-\frac{1}{8m}\gamma(\pi^*\mu)\psi^h
\end{equation*}
with parameters $(g^h_Y,\pi^*\mu)$ for every smooth map $h\colon X\rightarrow S^1$.
\end{prop}
\begin{proof}
If $\psi\equiv 0$, the statement is clear. If $\psi$ does not vanish identically on $Y$ we first apply the reverse implication of Theorem \ref{thm:main thm all n} to show that $(A_n,\phi)$ is a solution of the Seiberg--Witten equations for parameters $(g_X,\mu)$. Applying the gauge transformation and again Theorem \ref{thm:main thm all n} the claim follows. 
\end{proof}

\subsection{Charge conjugation}
Charge conjugation is an involution $\mathfrak{s}_X^c\mapsto \bar{\mathfrak{s}}_X^c$ that acts on the spinor bundle by complex conjugation,
\begin{equation*}
S_X^{c,n \pm}\longmapsto \bar{S}_X^{c,n \pm},\quad \phi\longmapsto \bar{\phi},
\end{equation*}
which is a complex anti-linear isomorphism of Clifford modules. It also maps
\begin{equation*}
\begin{array}{ccccl}
L_n&\longmapsto& \bar{L}_n& = &L_n^{-1}\\
A_n&\longmapsto & \bar{A}_n & = & -A_n\\
F_{A_n}&\longmapsto & F_{\bar{A}_n} & = & -F_{A_n}\\
\mu&\longmapsto & \bar{\mu} & = & -\mu.
\end{array}
\end{equation*}
The corresponding map on the charge $q$ is given by
\begin{equation*}
q\longmapsto -q,    
\end{equation*}
where the bundles $L$ and $L^{\scriptscriptstyle\frac{1}{2}}$, respectively, still have charge $1$, cf.~Remark \ref{rem:charge q}. 

For the lift of the charge conjugated pair $(\bar{A}_n,\bar{\phi})$ the data for the Kaluza--Klein circle bundle $(Y,A,g_Y,\mathfrak{s}_Y)$ stay the same and
\begin{equation*}
\begin{array}{rcl}
\psi=Q_{-q}(\phi)\in V_{-q}&\longmapsto& \bar{\psi}=Q_q(\bar{\phi})\in V_q\\
\pi^*\mu&\longmapsto &\pi^*\bar{\mu}\\
m&\longmapsto&-m.
\end{array}
\end{equation*}
Charge conjugation maps solutions $(A_n,\phi)$ to the Seiberg--Witten equations with parameters $(g_X,\mu)$ to solutions $(\bar{A}_n,\bar{\phi})$ with parameters $(g_X,\bar{\mu})$. This implies:
\begin{prop}
With the notation from Theorem \ref{thm:main thm all n} let
\begin{align*}
(\psi,g_Y^A,\pi^*\mu)&=\pi^*(A_n,\phi,g_X,\mu)\\
(\bar{\psi},g_Y^A,\pi^*\bar{\mu})&=\pi^*(\bar{A}_n,\bar{\phi},g_X,\bar{\mu})
\end{align*}
and $m=-q$. If $\psi$ is a solution to the equation
\begin{equation*}
D^Y\psi=-\frac{1}{16m}|\psi|^2\psi+m\psi-\frac{1}{8m}\gamma(\pi^*\mu)\psi
\end{equation*}
with parameters $(g_Y^A,\pi^*\mu)$, then $\bar{\psi}$ is a solution to the equation
\begin{equation*}
D^Y\bar{\psi}=\frac{1}{16m}|\bar{\psi}|^2\bar{\psi}-m\bar{\psi}+\frac{1}{8m}\gamma(\pi^*\bar{\mu})\bar{\psi}
\end{equation*}
with parameters $(g_Y^A,\pi^*\bar{\mu})$.
\end{prop}

\section{Kaluza--Klein circle bundles with fibres of length $2\pi r$}\label{sect:KK bundles fibres length 2pir}
The Kaluza--Klein metric can be generalized as follows: Let $r\colon X\rightarrow \mathbb{R}^+$ be a smooth positive function on $X$ and $\hat{r}=r\circ\pi$. Then
\begin{equation}\label{eqn:KK metric gen r}
g_Y=\pi^*g_X-\hat{r}^2A\otimes A
\end{equation}
is again a Riemannian metric on $Y$. The only difference to the original metric is that
\begin{equation*}
g_Y(K,K)=\hat{r}^2,
\end{equation*}
i.e.~the circle fibres over a point $x\in X$ have length $2\pi r(x)$. In the physics literature $r$ is written as $r=e^{\varphi}$ and the scalar field $\varphi\colon X\rightarrow \mathbb{R}$ is sometimes called the dilaton.

The vector field $\frac{1}{\hat{r}}K$ has unit length, hence the relation in Lemma \ref{lem:relation Clifford 4 and 5d} between Clifford multiplication on the spinor bundles is given by
\begin{align*}
V^*\cdot Q(\phi)&=Q(V\cdot\phi)\\
\tfrac{1}{\hat{r}}K\cdot Q(\phi)&=Q(i\mathrm{dvol}_{g_X}\cdot \phi).
\end{align*}
The formulas in Proposition \ref{prop:clifford identities S+} change accordingly to
\begin{align*}
\tfrac{1}{\hat{r}}K\cdot\psi&=-i\psi\\
\tfrac{1}{\hat{r}}K\cdot\pi^*\omega\cdot\psi&=-iQ(\omega_+\cdot\phi).
\end{align*}
The formula in Proposition \ref{prop:formula Dirac DY spin and non-spin} for the Dirac operator $D^Y$ then becomes (see \cite{AB} for the case of constant $r$ and \cite{Amm} for the general case)
\begin{align*}
D^Y&=Q\circ D_{A_n}^X\circ Q^{-1}-i\gamma\left(\tfrac{1}{\hat{r}}K\right)\left(\tfrac{1}{\hat{r}}q-\tfrac{1}{4}\hat{r}\gamma(\pi^*F_A)\right)\\
&=Q\circ \left(D_{A_n}^X-\frac{1}{\hat{r}}q+\frac{1}{8q}\hat{r}\gamma(F_{A_n}^+)\right)\circ Q^{-1},
\end{align*}
with $q$ as before. We then get:
\begin{cor}\label{cor:main cor all n all r}
Consider the Kaluza--Klein metric given by equation \eqref{eqn:KK metric gen r} with a smooth positive function $r\colon X\rightarrow \mathbb{R}$ and $\hat{r}=r\circ\pi$. Define a function $m_r\colon Y\rightarrow\mathbb{R}$ by $m_r=-\frac{q}{\hat{r}}$, where $q$ is the $\mathrm{U}(1)$-charge as before. Then $m_r$ is nowhere zero on $X$ (since $q\neq 0$ by assumption) and the statements in Theorem \ref{thm:main thm all n} and equation \eqref{eqn:DY equivalent SW eqns} continue to hold with $m$ replaced by $m_r$:
\begin{equation}\label{eqn:DY equivalent SW eqns gen r}
D^Y\psi=-\frac{1}{16m_r}|\psi|^2\psi+m_r\psi-\frac{1}{8m_r}\gamma(\pi^*\mu)\psi.
\end{equation}
\end{cor}
\begin{rem}\label{rem:GN}
Suppose that $\mu=0$ and the circle radius $r$ is constant. Then equation \eqref{eqn:DY equivalent SW eqns gen r}
\begin{equation*}
D^Y\psi=-\frac{1}{16m_r}|\psi|^2\psi+m_r\psi
\end{equation*} 
is the field equation of a $5$-dimensional Gross--Neveu model \cite{GN} (in Euclidean signature) with mass $m_r$ and coupling constant $g^2=\frac{1}{16m_r}$, given by the Lagrangian
\begin{equation}\label{eqn:GN Lagrangian}
\mathcal{L}[\psi]=\langle \psi,D^Y\psi\rangle -m_r|\psi|^2+\frac{1}{32m_r}|\psi|^4,
\end{equation}
where $\langle\cdot\,,\cdot\rangle$ is the Hermitian bundle metric on $S_Y$. Note that $m_r<0$ and $g^2<0$ if $q>0$. (The spinor $\phi$ on $X$ has charge $q$ and mass zero and the spinor $\psi$ on $Y$ has mass $m_r=-\frac{q}{r}$ and charge zero. If $r\ll |q|$, then the absolute value $|m_r|$ of the mass of $\psi$ is large, and if $r\gg |q|$, then $|m_r|$ is small. The interaction described by the term $\frac{1}{32m_r}|\psi|^4$ behaves in the opposite way.)
\end{rem}
\begin{rem}\label{rem:Ricci calc}
In Section \ref{sect:Sasaki} we need the following formulas: Suppose that $r>0$ is constant and $g_Y$ the Kaluza--Klein metric \eqref{eqn:KK metric gen r}. Then the circle fibres of $\pi\colon Y\rightarrow X$ are totally geodesic. For the curvature $2$-form $F_A\in\Omega^2(X,i\mathbb{R})$ and vectors $V,W\in TX$ define (see \cite[Chapter 9]{Besse})
\begin{align*}
\mathcal{A}_{V^*}W^*&=\frac{i}{2}F_A(V,W)K\\
(\mathcal{A}_{V^*},\mathcal{A}_{W^*})&=-\frac{1}{4}r^2\sum_{j=1}^4F_A(V,e_j)F_A(W,e_j)\\
(\mathcal{A}\tfrac{1}{r}K,\mathcal{A}\tfrac{1}{r}K)&=-\frac{1}{2}r^2\sum_{i<j}F_A(e_i,e_j)^2,
\end{align*}
where $\{e_i\}_{i=1}^4$ is an orthonormal frame in $T_pX$. The Ricci curvature of $g_Y$ then satisfies 
\begin{align*}
\mathrm{Ric}_{g_Y}(\tfrac{1}{r}K,\tfrac{1}{r}K)&=(\mathcal{A}\tfrac{1}{r}K,\mathcal{A}\tfrac{1}{r}K)\\
\mathrm{Ric}_{g_Y}(V^*,W^*)&=\mathrm{Ric}_{g_X}(V,W)-2(\mathcal{A}_{V^*},\mathcal{A}_{W^*}),\quad\forall V,W\in TX.
\end{align*}
If $F_A$ is coclosed and thus harmonic, we have in addition
\begin{equation*}
\mathrm{Ric}_{g_Y}(\tfrac{1}{r}K,V^*)=0\quad\forall V\in TX.
\end{equation*}
\end{rem}

\section{K\"ahler--Einstein $4$-manifolds and eigenspinors on Sasaki $5$-manifolds}\label{sect:Sasaki}
Spinors $\psi=Q(\phi)$ of constant length are an interesting case of equation \eqref{eqn:DY equivalent SW eqns gen r}, e.g.~for constant radius $r$ and perturbation $\mu=0$:
\begin{align*}
|\psi|^2&\equiv a^2,\quad\text{with $a>0$}\\
D^Y\psi&=\left(-\frac{1}{16m_r}a^2+m_r\right)\psi.
\end{align*}
Hence $\psi$ is an eigenspinor of $D^Y$.
\begin{rem}
In the special case that $a=4|m_r|$, the spinor $\psi$ is harmonic, $D^Y\psi=0$. The points $x=\pm 4|m_r|$ are the non-zero extrema of the potential 
\begin{equation*}
V(x)=-m_rx^2+\frac{1}{32m_r}x^4
\end{equation*}
appearing in the Gross--Neveu Lagrangian \eqref{eqn:GN Lagrangian} (which are minima for $m_r>0$ and maxima for $m_r<0$).
\end{rem}
Suppose that $\phi\in \Gamma(S_X^{c+})$ is an arbitrary spinor and $\psi=Q(\phi)$. Then $|\psi|\equiv a$ is equivalent to $|\phi|\equiv a$. Positive Weyl spinors $\phi$ of constant length are related to almost complex structures on $X$: Suppose that $J$ is a $g_X$-compatible almost complex structure on $X$, so that
\begin{equation*}
g_X(JV,JW)=g_X(V,W)\quad\forall V,W\in TX.
\end{equation*}
Then $(g_X,J,\omega)$, with self-dual fundamental $2$-form $\omega\in\Omega^2_+(X)$ defined by 
\begin{equation*}
\omega(V,W)=g_X(JV,W)\quad\forall V,W\in TX,
\end{equation*} 
is an almost Hermitian structure on $X$. There exists a canonical $\mathrm{Spin}^c$-structure $\mathfrak{s}_{X\mathrm{can}}^c$ on $X$ with spinor bundles
\begin{align*}
S_{X\mathrm{can}}^{c+}&=\Lambda^{0,0}\oplus \Lambda^{0,2}\\
S_{X\mathrm{can}}^{c-}&=\Lambda^{0,1}.
\end{align*}
We can write $S_{X\mathrm{can}}^{c+}=\underline{\mathbb{C}}\oplus K^{-1}$, where the canonical and anti-canonical line bundles are
\begin{equation*}
K=\Lambda^{2,0},\quad K^{-1}=\Lambda^{0,2}
\end{equation*} 
(we use the same symbol $K$ for the canonical bundle as for the unit Killing vector field in Section \ref{sect:kk circle bundles}; the meaning should be clear from the context). The characteristic line bundle of $\mathfrak{s}_{X\mathrm{can}}^c$ is $L=K^{-1}$. The spinor $\phi=(1,0)\in S_{X\mathrm{can}}^{c+}$ has constant length $|\phi|\equiv 1$. Conversely, every $\mathrm{Spin}^c$-structure $\mathfrak{s}_X^c$ with a positive spinor $\phi$ of constant length $1$ arises in this way for a $g_X$-orthogonal almost complex structure $J$ (see \cite{KM}, \cite{GS}).
\begin{lem}
Let $J$ be a $g_X$-compatible almost complex structure on $X$. The $\mathrm{Spin}^c$-structures $\mathfrak{s}_{X\mathrm{can}}^c$ and $\mathfrak{s}_{X\mathrm{can}}^c\otimes K$ have positive Weyl spinor bundles
\begin{align*}
S_{X\mathrm{can}}^{c+}&=\underline{\mathbb{C}}\oplus K^{-1}\\
S_{X\mathrm{can}}^{c+}\otimes K&=K\oplus\underline{\mathbb{C}}
\end{align*}
with characteristic line bundles $K^{-1}$ and $K$, respectively. The $\mathrm{Spin}^c$-structure $\mathfrak{s}_{X\mathrm{can}}^c\otimes K$ is the charge conjugate of $\mathfrak{s}_{X\mathrm{can}}^c$,
\begin{equation*}
S_{X\mathrm{can}}^{c}\otimes K=\bar{S}_{X\mathrm{can}}^{c}.
\end{equation*}
Clifford multiplication of the fundamental $2$-form $\omega$ on $S_{X\mathrm{can}}^{c+}$ is given by
\begin{equation}\label{eqn:Clifford action omega} 
\gamma(\omega)\left(\begin{array}{c}\alpha \\ \beta \end{array} \right)=-2i\left(\begin{array}{cc} 1 & 0 \\ 0 & -1 \end{array}\right)\left(\begin{array}{c}\alpha \\ \beta \end{array} \right)\quad\forall (\alpha,\beta)\in \underline{\mathbb{C}}\oplus K^{-1}.
\end{equation}
\end{lem}
For the expression for $\gamma(\omega)$ see e.g.~\cite[p.~112]{Morgan}. We consider the particular case \cite{W}, \cite{Don}, \cite{Morgan} where $(X,g_X,J,\omega)$ is a K\"ahler surface with integrable complex structure $J$, K\"ahler form $\omega$ (which is a self-dual, harmonic $2$-form) and compatible Riemannian metric $g_X$. In this situation explicit solutions to the Seiberg--Witten equations, with spinors of constant length, can be found. The following lemma summarizes some well-known facts about K\"ahler--Einstein manifolds, see e.g.~\cite{Besse}.
\begin{lem}\label{lem:KE curvature Chern class}
Let $(X,g_X,J,\omega)$ be a K\"ahler manifold. The Levi--Civita connection of $g_X$ defines Hermitian connections $A_{g_X}$ on $K$ and $A_{g_X}^{-1}$ on $K^{-1}$. If the metric $g_X$ is Einstein with Einstein constant $\lambda$, i.e.~
\begin{equation*}
\mathrm{Ric}_{g_X}=\lambda g_X,
\end{equation*}
or equivalently $\rho=\lambda\omega$ with the Ricci form $\rho$, then the curvature $2$-forms of these connections are self-dual and given by
\begin{equation*}
F_{A_{g_X}}=i\lambda\omega\quad\text{and}\quad F_{A_{g_X}^{-1}}=-i\lambda\omega.
\end{equation*}
We have
\begin{equation*}
c_1(K)=\frac{i}{2\pi}[F_{A_{g_X}}]=-\frac{\lambda}{2\pi}[\omega]=-c_1(K^{-1}).    
\end{equation*}
\end{lem}
\begin{prop}\label{prop:KE sol SW}
Let $(X,g_X,J,\omega)$ be a closed K\"ahler--Einstein surface with Einstein constant $\lambda$. We consider the Seiberg--Witten equations
\begin{equation}\label{eqn:KE SW eqns pertub}
\begin{split}
D_{A}^X\phi&=0\\
F_{A}^+&=\sigma(\phi,\phi)+it\omega
\end{split}
\end{equation}
with perturbation $\mu=it\omega$, where $t\in\mathbb{R}$.
\begin{enumerate}
\item For the $\mathrm{Spin}^c$-structure $\mathfrak{s}_{X\mathrm{can}}^c$, the connection $A_0=A_{g_X}^{-1}$ on $K^{-1}$ and the spinor
\begin{equation*}
\phi_0=(2\sqrt{-\lambda-t},0)\in \Gamma(\underline{\mathbb{C}}\oplus K^{-1})
\end{equation*} 
are a solution to the perturbed Seiberg--Witten equations \eqref{eqn:KE SW eqns pertub} for all $t<-\lambda$. We have $\gamma(\mu)\phi_0=2t\phi_0$.
\item For the $\mathrm{Spin}^c$-structure $\mathfrak{s}_{X\mathrm{can}}^c\otimes K$, the connection $A_0=A_{g_X}$ on $K$ and the spinor
\begin{equation*}
\phi_0=(0,2\sqrt{-\lambda+t})\in \Gamma(K\oplus\underline{\mathbb{C}})
\end{equation*} 
are a solution to the perturbed Seiberg--Witten equations \eqref{eqn:KE SW eqns pertub} for all $t>\lambda$. We have $\gamma(\mu)\phi_0=-2t\phi_0$.
\end{enumerate}
\end{prop}
\begin{proof}
We first consider the case with $\mathrm{Spin}^c$-structure $\mathfrak{s}_{X\mathrm{can}}^c$. For a general K\"ahler surface, the Dirac operator on $\Gamma(S_{X\mathrm{can}}^{c+})$ associated to the connection $A_0=A_{g_X}^{-1}$ on $K^{-1}$ is given by \cite{Hit} 
\begin{equation*}
D_{A_0}^X(\alpha,\beta)=\sqrt{2}(\bar{\partial}\alpha+\bar{\partial}^*\beta)\quad\forall (\alpha,\beta)\in \Gamma(\underline{\mathbb{C}}\oplus K^{-1}).
\end{equation*}
For a spinor $\phi_0=(\alpha,0)\in \Gamma(\underline{\mathbb{C}}\oplus K^{-1})$ we have by equation \eqref{eqn:Clifford action sigma phi phi}
\begin{equation*}
\gamma(\sigma(\phi_0,\phi_0))=\frac{1}{2}|\alpha|^2\left(\begin{array}{cc} 1 & 0 \\ 0 & -1  \end{array}\right).
\end{equation*}
If $(X,J,\omega,g_X)$ is K\"ahler--Einstein with Einstein constant $\lambda$, then by equation \eqref{eqn:Clifford action omega} and Lemma \ref{lem:KE curvature Chern class}
\begin{equation*}
\gamma(it\omega)=2t\left(\begin{array}{cc} 1 & 0 \\ 0 & -1 \end{array}\right),\quad \gamma(F_{A_0}^+)=-2\lambda\left(\begin{array}{cc} 1 & 0 \\ 0 & -1 \end{array}\right).
\end{equation*}
If $\alpha\in\mathbb{R}$ is a constant, the first equation in \eqref{eqn:KE SW eqns pertub} is satisfied and the second equation reduces to $-2\lambda=\frac{1}{2}\alpha^2+2t$. This implies the claim in the first case.

For the $\mathrm{Spin}^c$-structure $\mathfrak{s}_{X\mathrm{can}}^c\otimes K$ with connection $A_0=A_{g_X}$ on $K$ and spinor $\phi_0=(0,\beta)\in\Gamma(K\oplus\underline{\mathbb{C}})$ it follows analogously that
\begin{align*}
\gamma(\sigma(\phi_0,\phi_0))&=\frac{1}{2}|\beta|^2\left(\begin{array}{cc} -1 & 0 \\ 0 & 1  \end{array}\right)\\
\gamma(it\omega)&=2t\left(\begin{array}{cc} 1 & 0 \\ 0 & -1 \end{array}\right),\quad \gamma(F_{A_0}^+)=2\lambda\left(\begin{array}{cc} 1 & 0 \\ 0 & -1 \end{array}\right),
\end{align*}
where the third equation holds in the K\"ahler--Einstein case. If $\beta\in\mathbb{R}$ is a constant, the first equation in \eqref{eqn:KE SW eqns pertub} is satisfied and the second equation reduces to $2\lambda=-\frac{1}{2}\beta^2+2t$. This implies the claim in the second case.
\end{proof}
\begin{rem}
In particular, if $\lambda\geq 0$, then the perturbation $it\omega$ has to be chosen non-zero.
\end{rem}
The following Boothby-Wang construction of Sasaki structures on principal circle bundles over K\"ahler manifolds, whose K\"ahler form represents an integral class up to a real multiple, is well-known \cite{Hat}.
\begin{lem}
Suppose that $(X,g_X,J,\omega)$ is a K\"ahler manifold so that $\frac{1}{\pi r}[\omega]$ lies in the image of $H^2(X;\mathbb{Z})$ in $H^2_{dR}(X)$ for some $r\in\mathbb{R}^+$. Let $\pi\colon Y\rightarrow X$ be the principal circle bundle with Euler class $e(Y)=\pm\frac{1}{\pi r}[\omega]$ and $A$ a connection on $Y$ with curvature $2$-form satisfying $\frac{i}{2\pi}F_A=\pm\frac{1}{\pi r}\omega$. Then $\eta=\pm irA$ satisfies $d\eta=2\pi^*\omega$ and together with the Kaluza--Klein metric 
\begin{equation*}
g_Y=\pi^*g_X-r^2A\otimes A=\pi^*g_X+\eta\otimes \eta
\end{equation*}
defines a Sasaki structure on $Y$.
\end{lem}
We apply this construction in the case of K\"ahler--Einstein surfaces, where $\frac{\lambda}{2\pi}\omega$ represents an integral cohomology class.
\begin{prop}\label{prop:KE circle bundle Sasaki} Suppose that $(X,g_X,J,\omega)$ is a closed K\"ahler--Einstein surface with Einstein constant $\lambda\neq 0$. Let $\pi\colon Y\rightarrow X$ be one of the two Kaluza--Klein circle bundles with Euler class $e(Y)$, connection $A$, constant radius $r$ and Kaluza--Klein metric
\begin{equation*}
g_Y=\pi^*g_X-r^2A\otimes A
\end{equation*}
chosen as in Table \ref{table:circle bundle Y over X KE}.
\begin{table}
{\footnotesize
\caption{Kaluza--Klein circle bundle $\pi\colon Y^5\rightarrow X^4$ over K\"ahler--Einstein surface $X$}
\label{table:circle bundle Y over X KE}
\renewcommand{\arraystretch}{1.5}
\setlength{\tabcolsep}{0.50em}
\begin{tabular}{lcccccccc} 
\hline\noalign{\smallskip}
$X$ & $e(Y)$ & $A$ & $r$ & $\mathfrak{s}_Y$ & $S_{X}^{c,n}$ & $L_n$ & $q$ & $m_r$ \\ 
\noalign{\smallskip}\hline\noalign{\smallskip}
spin or non-spin & $c_1(K^{-1})$ & $A_{g_X}^{-1}$ & $\frac{2}{|\lambda|}$  & odd &  $S_{X\mathrm{can}}^c$, $S_{X\mathrm{can}}^c\otimes K$ & $K^{-1}$, $K$ & $\tfrac{1}{2}$, $-\tfrac{1}{2}$ & $-\tfrac{1}{4}|\lambda|$, $\tfrac{1}{4}|\lambda|$\\ 
spin & $c_1(K^{-\scriptscriptstyle\frac{1}{2}})$ & $\frac{1}{2}A_{g_X}^{-1}$ & $\frac{4}{|\lambda|}$ & even &  $S_{X\mathrm{can}}^c$, $S_{X\mathrm{can}}^c\otimes K$ & $K^{-1}$, $K$  & $1$, $-1$ & $-\tfrac{1}{4}|\lambda|$, $\tfrac{1}{4}|\lambda|$\\
\noalign{\smallskip}\hline
\end{tabular}
}
\end{table}
A Sasaki structure on $Y$ is defined by $g_Y$ and the $1$-form $\eta=\tfrac{2}{\lambda}iA_{g_X}^{-1}$ on $Y$, satisfying
\begin{equation*}
g_Y=\pi^*g_X+\eta\otimes\eta,\quad d\eta=2\pi^*\omega.
\end{equation*}
Then
\begin{equation*}
\mathrm{Ric}_{g_Y}=(\lambda-2)g_Y+(6-\lambda)\eta\otimes \eta,\quad \mathrm{scal}_{g_Y}=4(\lambda-1),
\end{equation*}
hence the Sasaki structure is Sasaki $\eta$-Einstein. Furthermore, the $5$-manifold $Y$ is spin and a lift $\mathfrak{s}_Y$ of the $\mathrm{Spin}^c$-structure $\mathfrak{s}_{X\mathrm{can}}^c$ is chosen as in Definition \ref{defn:main construction Y from X spinc}.
\end{prop}
\begin{proof}
See \cite{BGM}, \cite{FOW} for the definition of Sasaki $\eta$-Einstein structures. Note that in both cases $\eta=\pm irA$, depending on whether the sign of $\lambda$ is $\pm$. The formula for the Ricci curvature follows from Remark \ref{rem:Ricci calc}, because the K\"ahler form $\omega$ is harmonic and $|\omega|^2=2$, hence for both cases
\begin{equation*}
(\mathcal{A}\tfrac{1}{r}K,\mathcal{A}\tfrac{1}{r}K)=4,\quad (\mathcal{A}_{V^*},\mathcal{A}_{W^*})=g_X(V,W).
\end{equation*}
\end{proof}
\begin{rem}\label{rem:Sasaki-Einstein}
Suppose that the Einstein constant $\lambda$ is positive and define $g'_X=cg_X$ for a constant $c\in\mathbb{R}^+$. Then $g'_X$ is K\"ahler--Einstein with $\lambda'=\frac{1}{c}\lambda$. In particular, $g_X$ can be normalized such that $\lambda=6=\dim X+2$. Then $g_Y$ is a Sasaki--Einstein metric with $\mathrm{Ric}_{g_Y}=4g_Y$.
\end{rem}
\begin{thm}
Let $(X,g_X,J,\omega)$ be a closed K\"ahler--Einstein surface with Einstein constant $\lambda\neq 0$ and $\pi\colon Y\rightarrow X$ one of the two Kaluza--Klein circle bundles with Sasaki $\eta$-Einstein structure and spin structure $\mathfrak{s}_Y$ given by Proposition \ref{prop:KE circle bundle Sasaki}.

Consider the solutions $(A_0,\phi_0)$, with $|\phi_0|\equiv\mathrm{const}.$, of the Seiberg--Witten equations on $X$ with perturbation $it\omega$ (for a suitable $t\in\mathbb{R}$) given by Proposition \ref{prop:KE sol SW}. Define
\begin{equation*}
\nu_\lambda=-\tfrac{1}{4}|\lambda|-\mathrm{sgn}(\lambda)=\begin{cases} -\frac{1}{4}|\lambda|+1& \text{if } \lambda<0 \\ -\frac{1}{4}\lambda-1 & \text{if } \lambda>0\end{cases}    
\end{equation*}
\begin{enumerate}
\item For the solution $\phi_0\in\Gamma(S_{X\mathrm{can}}^{c+})$ the lifted spinor $\psi=Q(\phi_0)\in \Gamma(S_Y)$ satisfies
\begin{align*}
|\psi|^2&\equiv 4(-\lambda-t)\\
D^Y\psi&=\nu_\lambda\psi.
\end{align*}
\item For the solution $\phi_0\in\Gamma(S_{X\mathrm{can}}^{c+}\otimes K)$ the lifted spinor $\psi=Q(\phi_0)\in \Gamma(S_Y)$ satisfies
\begin{align*}
|\psi|^2&\equiv 4(-\lambda+t)\\
D^Y\psi&=-\nu_\lambda\psi.
\end{align*}
\end{enumerate}
\end{thm}
\begin{proof}
This follows from
\begin{equation*}
|\psi|^2\equiv a^2\quad\text{and}\quad D^Y\psi=\left(-\frac{1}{16m_r}a^2+m_r-\frac{1}{8m_r}\gamma(\pi^*\mu)\right)\psi
\end{equation*}
with $a$ and $\gamma(\mu)\phi_0$ from Proposition \ref{prop:KE sol SW} and $m_r$ from Table \ref{table:circle bundle Y over X KE}.
\end{proof}
\begin{rem}
The eigenvalues can also be written as $\nu_\lambda = m_r-\mathrm{sgn}(\lambda)$ in the first case and $-\nu_\lambda = m_r+\mathrm{sgn}(\lambda)$ in the second case, which relates them to the mass $m_r$ of the spinor $\psi$.
\end{rem}
\begin{rem}
If $\lambda<0$ and $g_X$ is normalized such that $\lambda =-4$, the spinor $\psi$ is harmonic in both cases, $D^Y\psi=0$.
\end{rem}
\begin{rem}
Suppose that $\lambda>1$. Then the scalar curvature $\mathrm{scal}_{g_Y}=4(\lambda-1)$ of $(Y,g_Y)$ is positive and the lifted spinors $\psi\in\Gamma(S_Y)$ are eigenspinors of the Dirac operator $D^Y$ with eigenvalue
\begin{equation*}
\nu_{\pm}=\pm\left(\frac{1}{4}\lambda+1\right).
\end{equation*}
According to a theorem of Friedrich \cite{F_EW} the eigenvalues $\nu_{\pm}$ have to satisfy
\begin{equation}\label{eqn:Friedrich ineq}
\nu_{\pm}^2\geq \frac{\dim Y}{4(\dim Y-1)}\mathrm{scal}_{g_Y}
\end{equation}
i.e.~in our situation
\begin{equation*}
\left(\frac{1}{4}\lambda+1\right)^2-\frac{5}{4}(\lambda-1)\geq 0.
\end{equation*}
This inequality holds, because the left-hand side is equal to the square $\frac{1}{16}(\lambda-6)^2$. Equality in \eqref{eqn:Friedrich ineq} holds if and only if $\lambda=6$, i.e.~$(Y,g_Y)$ is a Sasaki--Einstein manifold by Remark \ref{rem:Sasaki-Einstein}. In this case the eigenvalues are $\nu_{\pm}=\pm\frac{5}{2}$ and by \cite{F_EW} the spinors $\psi$ are Killing spinors,
\begin{equation*}
\nabla^Y_V\psi=\mp\frac{1}{2}V\cdot \psi\quad\forall V\in TY.
\end{equation*}
The existence of two linearly independent Killing spinors on complete, simply connected Sasaki--Einstein $5$-manifolds is well-known \cite{FK}, \cite{Baer_KS} (by the Bonnet--Myers Theorem the fundamental group of any complete Sasaki--Einstein manifold is finite). 

Lower bounds on the eigenvalues of $D^Y$ have also been proved in \cite[Theorem 6.1]{Kim} for the case of $\mathrm{scal}_{g_Y}> -4$, i.e.~Einstein constant $\lambda>0$, which for certain values of $\lambda$ are improvements of the estimate \eqref{eqn:Friedrich ineq}. It can be checked that $\nu_{\pm}=\pm\left(\frac{1}{4}\lambda+1\right)$ satisfy these bounds as well.  
\end{rem}

\section*{Acknowledgments}
I would like to thank the anonymous referee for a number of comments and corrections that helped to improve the quality of the paper. I would also like to thank Bernd Ammann and Uwe Semmelmann for discussions on a preprint version.

\bibliography{SWLift5MfdsBib}
\bibliographystyle{plain}

\end{document}